\documentclass{amsart}
\usepackage{amsmath,amsthm,amssymb,amsfonts,ifpdf,bbm,verbatim}
\usepackage{enumitem}
\usepackage{graphicx}
\usepackage[usenames,dvipsnames]{color}
\usepackage{tensor,calc}

\ifpdf
\usepackage[pdftex,pdfstartview=FitH,pdfborderstyle={/S/B/W 1},%
colorlinks=true, linkcolor=blue, urlcolor=blue, citecolor=blue,%
pagebackref=true]{hyperref}
\else
\usepackage[dvips]{hyperref}
\fi
\numberwithin{equation}{section}

\usepackage[shortalphabetic,abbrev,msc-links,nobysame]{amsrefs}

\newtheorem {Theorem}    {Theorem}[section]
\newenvironment{Theorem*}
  {\Theorem}
  {\endTheorem}
\newtheorem {Lemma}      [Theorem]    {Lemma}

\newtheorem {Corollary}  [Theorem]    {Corollary}
\newtheorem {Proposition}[Theorem]    {Proposition}

\theoremstyle{definition}

\newtheorem {Exercise}   {Exercise}
\newenvironment{Exercise*}
  {\Exercise}
  {\endExercise}

\newcounter{AbcT}

\numberwithin{equation}{section}

\newcommand {\Heads}[1]   {\smallskip\pagebreak[1]\noindent{\bf #1{\hskip 0.2cm}}}
\newcommand {\Head}[1]    {\Heads{#1:}}

\DeclareMathOperator {\E} {{\mathbf{E}}}

\newcommand {\Q} {{\mathbb Q}}
\newcommand {\R} {{\mathbb R}}

\newcommand {\Z} {{\mathbb Z}}

\newcommand {\cC} {{\mathcal C}}

\renewcommand{\liminf}{\varliminf}
\renewcommand{\limsup}{\varlimsup}
\DeclareMathOperator{\supp}{supp}

\DeclareMathOperator{\SL}{SL}

\renewcommand {\sl} {\operatorname{sl}}

\DeclareMathOperator{\GL}{GL}

\newcommand {\IGNORE}[1]  {}

\renewcommand {\setminus}       {\smallsetminus}

\newcommand {\absolute}[1] {\left| {#1} \right|}
\newcommand {\norm}[1] {\left\| {#1} \right\|}

\renewcommand{\sl}{\mathfrak{sl}}

\newcommand\cA{\mathcal{A}}

\newcommand\Ad{\operatorname{Ad}}

\DeclareMathOperator{\Lie}{Lie}

\newcommand{\localfield}{k}
\newcommand{\invmeas}[2]{{\mu^{#1 }_{\mathrm{inv},#2}}}
\newcommand{\Dinv}{D^{\mathrm{inv}}_\mu}

\begin{document}

	\title[Symmetry of entropy in higher rank diagonalizable actions]{Symmetry of entropy in higher rank diagonalizable actions and measure classification}
	\author[M. Einsiedler]{Manfred Einsiedler}
	\address[M. E.]{ETH Z\"urich, R\"amistrasse 101
	CH-8092 Z\"urich
	Switzerland}
	\email{manfred.einsiedler@math.ethz.ch}
	\author[E. Lindenstrauss ]{Elon Lindenstrauss}
	\address[E. L.]{The Einstein Institute of Mathematics\\
	Edmond J. Safra Campus, Givat Ram, The Hebrew University of Jerusalem
	Jerusalem, 91904, Israel}
	\email{elon@math.huji.ac.il}
	\thanks{M.~E.~acknowledges the support by the SNF (Grant 200021-152819).
	E.~L.~acknowledges the support of ISF grant 891/15.}
	\date{\today}
	\dedicatory{In memory of Roy Adler}
	\begin{abstract}
	An important consequence of the theory of entropy of $\Z$-actions is that the events measurable with respect to the far future coincide (modulo null sets) with those measurable with respect to the distant past, and that measuring the entropy using the past will give the same value as measuring it using the future. In this paper we show that for measures invariant under multiparameter algebraic actions if the entropy attached to coarse Lyapunov foliations fail to display a stronger symmetry property of a similar type this forces the measure to be invariant under non-trivial unipotent groups. Some consequences of this phenomenon are noted.
	\end{abstract}
	\maketitle

\section{Introduction}\label{introduction}
Let $M = \prod_ {\ell = 1 } ^ m \SL (d, \localfield _ \ell)$ with $\localfield _ \ell$ local fields of either zero or positive characteristic (not necessarily the same for all $\ell$). Let $G$ be a closed subgroup of $M$, $\Gamma$ a lattice in $G$ and $a_1, a_2, \dots, a _ r $ be $r$ elements in $G$ so that for every $i$, all of the components of $a_i$ in $\SL (d, \localfield _ \ell)$ are diagonal matrices; in particular all the $a _ i$ commute. For $\mathbf n \in \Z ^ r$ we denote $a ^ \mathbf n = a_1 ^ {n _ 1} \dots a _ r ^ {n _ r}$.

Suppose $1 \leq \ell \leq m$ and for distinct $1 \leq i,j \leq d$ let $E _ {ij} ^ {\ell}$ denote the elementary unipotent subgroup of $\SL (d, \localfield _ \ell)$ with one on the diagonal, arbitrary element of $\localfield _ \ell$ on the $i, j$ entry and zero elsewhere.
Then there is some linear functional $\alpha: \Z ^ r \to \R$ so that for every $h \in E_{ij}^{\ell}$
\begin{equation*}
\absolute {a ^ {\mathbf n} (h-1) a ^ {- \mathbf n}} = e ^ {\alpha (\mathbf n)} \absolute {h-1}
.\end{equation*} 
These functionals will be called the \emph{Lyapunov exponents} of the action of $a$ on $M$, and the set of such functionals will be denoted by $\Phi$. Two Lyapunov exponents $\alpha, \alpha ' \in \Phi$ will be said to correspond to the same coarse exponent if $\alpha = c \alpha '$ for some $c > 0$. The equivalence class of a Lyapunov exponent $\alpha$ under this equivalence relation will be denoted by $[\alpha]$, and the set of equivalence classes, a.k.a.\ the \emph{coarse Lyapunov exponents} will be denoted by $[\Phi]$.

For every $\mathbf n \in \Z ^ r$ we define the corresponding expanding horospherical subgroup $G ^ + _ {\mathbf n}$ of $G$ by
\begin{equation*}
G ^ + _ {\mathbf n} = \left\{ g \in G: a ^ {-j \mathbf n} g a ^ {j \mathbf n} \to 1 \text{ as $j \to \infty$}\right\} .\end{equation*}
We can attach a subgroup $U _ {[\alpha]}$ to every coarse Lyapunov exponent $[\alpha] \in [\Phi]$ as follows:
\begin{equation*}
U _ {[\alpha]} = \bigcap_ {\mathbf n: \alpha (\mathbf n) > 0} G ^ + _ {\mathbf n}
.\end{equation*}
Somewhat more explicitly, let $M _ {[\alpha]}$ denote the subgroup of $M$ generated by all elementary one parameter unipotent groups $E _ {i j} ^ \ell$ for which the corresponding Lyapunov exponent is in $[\alpha]$. Then it can be shown that $U _ {[\alpha]} = G \cap M _ {[\alpha]}$. Note that our definitions imply that if $\alpha \in \Phi$ so is $-\alpha$, but it may well happen for some $G$ that $U_{[\alpha] }$ is nontirivial but $U_{[-\alpha] } = \{1\}$. 
We will say that a coarse Lyapunov exponent $[\alpha]$ is \emph{a coarse Lyapunov exponent for~$G$} (or appears in $G$) if $U_{[\alpha]}\neq\{1\}$.

Let $\mu$ be an $A$-invariant and ergodic probability measure on $G/\Gamma$.
To each coarse Lyapunov exponent $[\alpha] \in [\Phi]$, we attach a system of leafwise measures $\mu _ x ^ {[\alpha]}$. Formally, $x \mapsto \mu _ x ^ {[\alpha]}$ is a Borel measurable map from $G / \Gamma$ to the space of equivalence classes up to a positive multiplicative constant of locally finite measures on $U _ {[\alpha]}$ satisfying a suitable growth condition (enforcing such a growth condition makes the space of locally finite measures up to a multiplicative constants into a compact metrizable space).

Let $I ^ {[\alpha]} _ x$ be the group of $u \in U _ {[\alpha]}$ satisfying that $\mu _ x ^ {[\alpha]}u = \mu _ x ^ {[\alpha]}$. As a locally compact nilpotent group, $I^{[\alpha]}_x$ is unimodular. It would be convenient to have also a notation for the Haar measure on $I ^ {[\alpha]} _ x$ considered as a locally finite measure up to multiplicative constant --- we denote this measure by $\invmeas {[\alpha]} x$.
Using Poincare recurrence and ergodicity one can easily show that if
$I _ x ^ {[\alpha]}$ is nontrivial on a set of positive $\mu$-measure then it is nontrivial a.s. and moreover contains arbitrarily small and arbitrary large elements of $U _ {[\alpha]}$, i.e. the group 
$I _ x ^ {[\alpha]}$ is neither discrete nor bounded.

We define for $\mathbf n \in \Z^r$ and $\alpha \in \Phi$ with $\alpha(\mathbf n)>0$ the \emph{entropy contribution of $\alpha$ for $\mathbf n$}, denoted by $D_\mu(\mathbf n,[\alpha])$, and the \emph{entropy contribution from the invariance group} $\Dinv (\mathbf n, [\alpha])$ by
\begin{equation}
\begin{aligned}\label{equation defining entropy contribution}
D_\mu(\mathbf n,[\alpha]) &= \lim_ { \ell \to \infty} \frac {-\log \mu ^ {[\alpha]} _ x (a^{-\ell \mathbf n} \Omega _ 0 a^{\ell \mathbf n}) }{ \ell} \\
\Dinv (\mathbf n, [\alpha]) &= \lim_ { \ell \to \infty} \frac {-\log \invmeas {[\alpha]} x (a^{-\ell \mathbf n} \Omega _ 0 a^{\ell \mathbf n}) }{ \ell}
,\end{aligned}
\end{equation}
where 
$\Omega _ 0$ is a relatively compact open neighborhood of $1$ in $U _ {[\alpha]}$. 
For notational convenience, we set $D_\mu(\mathbf n,[\alpha])=\Dinv (\mathbf n, [\alpha])=0$ if $\alpha(\mathbf n)\leq 0$. Formally, these quantities depend on the choice of $x \in G/\Gamma$ and strictly speaking to make sense of the expressions inside the limits above one needs to choose a particular measure in the proportionallity class $\mu _ x ^ {[\alpha]}$ and $\invmeas {[\alpha]} x$. Both limits in \eqref{equation defining entropy contribution} are known to exist and moreover
\begin{equation*}
\frac {D _ \mu (\mathbf n, [\alpha]) }{ \alpha (\mathbf n)} =
\frac {D _ \mu (\mathbf m, [\alpha]) }{ \alpha (\mathbf m)}
\end{equation*}
for every $\mathbf n, \mathbf m \in \Z ^ r$ with $\alpha (\mathbf n), \alpha (\mathbf m) > 0$ (and similarly for $\Dinv(\bullet, [\alpha])$) --- see e.g.~\cite{Einsiedler-Lindenstrauss-Clay}. We will write $a^{\bullet}$ for the group $\{a^{\mathbf{n}}:\mathbf n \in \Z^r\}$.

\begin{Theorem}\label{symmetry theorem}
Let $\mu$ be an $a ^ {\bullet}$-invariant and ergodic measure on $G / \Gamma$, with $a ^ {\bullet}$, $G$ and $\Gamma$ as above. Let $[\alpha] \in [\Phi]$ and $\mathbf n \in \Z ^ r$ satisfy $\alpha (\mathbf n) > 0$.
 Then
\begin{equation*}
\Dinv (\mathbf n,{[\alpha]}) \geq D _ \mu (\mathbf n,{[\alpha]}) - D _ \mu (-\mathbf n,{[-\alpha]})
.\end{equation*}
In particular, if $ D _ \mu (\mathbf n,{[\alpha]}) > D _ \mu (-\mathbf n,{[-\alpha]})$ the invariance group $I _ x ^ {[\alpha]}$ contains arbitrarily small and arbitrary large elements. Moreover, if $D _ \mu (\mathbf n,{[-\alpha]})=0$, and in particular if $U _ {[- \alpha]} = \left\{ 1 \right\}$, then $\mu _ x ^ {[\alpha]} = \invmeas {[\alpha]} x$ a.s.
\end{Theorem}

For $\alpha$ for which $U _ {[- \alpha]} = \left\{ 1 \right\}$, Theorem~\ref{symmetry theorem} essentially reduces to the main technical result of Katok and Spatzier's paper \cite{Katok-Spatzier} (see also \cite{Kalinin-Katok-Seattle}). 
The symmetry of the entropy contribution has been an important component in the classification of measures invariant under a maximal split torus in semisimple groups, e.g.\ \cite[Cor.~3.4]{Einsiedler-Katok-Lindenstrauss}, \cite[Thm.~5.1]{Einsiedler-Lindenstrauss-split}. In those papers, this symmetry was derived from more detailed analysis and using additional structure; the main point of this paper is that this symmetry is a rather general feature of higher rank diagonal actions. Thus one can view Theorem~\ref{symmetry theorem} as a common generalization of both the techniques of \cite{Katok-Spatzier} and the above auxiliary results from \cite{Einsiedler-Katok-Lindenstrauss,Einsiedler-Lindenstrauss-split}.

Specializing further, we obtain a sharper result in the same vein. Suppose now all the local fields $\localfield _ \ell$ are either $\R$ or $\Q _ {p}$ (of course, more than one prime $p$ may be used). Suppose further that $a _ 1, a _ 2, \dots, a _ r$ are not only diagonal but satisfy the following further assumption:
\begin{description}[align=left,style=nextline,leftmargin=*,labelsep=\parindent,font=\normalfont\itshape]
\item [Class-$\cA'$] The components of all $a _ i$ over $\R$ are \emph{positive} diagonal matrices, and for every $\Q _ p$ there is some $\theta _ p \in \Q _ p ^ {\times}$ with $\absolute {\theta _ p} > 1$ so that all the  entries in the diagonal of the $\Q _ p$-components of all the $a _ i$ are powers of $\theta _ p$.
\end{description}

\begin{Theorem}\label{refined main theorem} Let $G$ be as above, $\Gamma < G$ a discrete subgroup, and assume that $a _ 1$, \dots, $a _ r$ satisfy the Class-$\cA'$ assumption. Let $\mu$ be an $a ^ {\bullet}$-invariant and ergodic measure on $G / \Gamma$. Then there is a closed subgroup $L < G$ containing the group $\{a^\bullet\}$, an element $g_0 \in G$ and a closed normal subgroup (possibly trivial) $H \lhd L$ so that $\mu$ is $H$-invariant and supported on the single $L$-orbit $L.[g_0]_\Gamma$, \ $g_0 ^{-1} H g_0 \cap \Gamma$ is a lattice in $g_0 ^{-1} H g_0 $, and if $\pi: L \to L / H$ is the natural projection $\Lambda = \pi (g_0 \Gamma g_0^{-1} \cap L)$ is a discrete subgroup of $L / H$. Moreover, the corresponding $\bar a _ 1 = \pi (a_1)$, \dots, $\bar a _ 1 = \pi (a_1)$ invariant probability measure $\bar \mu$ on $(L/H) / \Lambda$ satisfies that for each coarse Lyapunov exponent $[\alpha] \in [\Phi]$ and for every $\mathbf n \in \Z^r$,
\begin{equation} \label{quotient symmetry equation}
D _ {\bar\mu} (\mathbf n, [\alpha]) = D _ {\bar\mu} (-\mathbf n, [-\alpha])
.\end{equation}
\end{Theorem}

In \eqref{quotient symmetry equation}, the entropy contributions $D _ {\bar\mu} (\mathbf n, [\alpha])$ are defined as above using the leafwise measures for the action of $\pi(U_{[\alpha]})$ on $(L/H) / \Lambda$.

To illustrate better the implications of Theorems~\ref{symmetry theorem} and~\ref{refined main theorem} we consider the action of the full diagonal group $A < \SL (n, \R)$ on $\SL (n, \R) \ltimes \R ^ n / \SL (n, \Z) \ltimes \Z ^ n $.

\begin{Theorem} \label{affine group theorem} Let $G = \SL (n, \R) \ltimes \R ^ n$ and $\Gamma = \SL (n, \Z) \ltimes  \Z ^ n$, and let $A$ be the maximal diagonalizable subgroup of $\SL (n, \R) < G$ for $n\geq3$.
Let $\mu$ be an $A$-invariant and ergodic measure on $G / \Gamma$ such that for some $a \in A$ the ergodic theoretic entropy $h _ \mu (a)$ is positive. 
Then either $\mu$ is homogeneous or $\mu$ is an extension of a zero entropy $A$-invariant measure $\bar \mu$  on $\SL (n, \R) / \SL (n, \Z)$ (i.e.~$h _ {\bar \mu} (a) = 0$ for any $a \in A$) with Haar measure on the fibers of the extension $G / \Gamma \to \SL (n, \R) / \SL (n, \Z)$.
\end{Theorem}

In addition to Theorem~\ref{symmetry theorem}, the proof of this theorem uses a measure classification results by A. Katok and the two authors of this paper \cite{Einsiedler-Katok-Lindenstrauss} and a result from our paper \cite{Einsiedler-Lindenstrauss-joinings-2}. 
We note that the proof of Theorem~\ref{refined main theorem} also uses a previous measure classification result --- the classification of invariant measure under groups generated by one-parameter unipotent subgroups in the setting of products of real and $p$-adic linear algebraic groups by Ratner~\cite{Ratner-padic} and Margulis and Tomanov~\cite{Margulis-Tomanov}, extending Ratner's measure classification theorem in the real case \cite{Ratner-Annals}.

\section{Preliminaries on leafwise measures}

We recall some basic facts about the construction of leafwise measures.
These are defined in \cite [\S6]{Einsiedler-Lindenstrauss-Clay} in the following general setup: let $X$ be a locally compact second countable metric space, and $U$ a unimodular locally compact second countable group equipped with a proper right invariant metric. Let $B^U_r(u)$ denotes the open ball of radius $r$ around $u \in U$, and $B_r^U=B_r^U(1)$. We assume $U$ acts continuously on $X$ (i.e. the map $(u,x) \mapsto u.x$ is a continuous map $U \times X \to X$) which is locally free, i.e.~for every compact $K \subset X$ there is a $\delta > 0$ so that for all $x \in K$ the map $u \mapsto u .x$ is injective on $B ^ U _ \delta$. Let $\lambda_U$ denote the Haar measure on $U$ normalized so that $\lambda_U(B_1^U) = 1$. 

Given a strictly positive function $\rho$ on $U$ we can consider the space $PM^*_\infty (U)$ of equivalence classes under proportionality of Radon measures $\vartheta $ on $U$ for which $\int_ U \rho (u) \,d \vartheta (u) < \infty$. For a locally compact second countable group $U$ one can introduce a metric on this space under which it is relatively compact.  
Indeed, one may take a sequence $f _ i \in C_c(U)$ spanning a dense subset of $C _ c (U)$, and define
\begin{equation*}
d(\nu, \nu ') =
\sum_ i
2 ^ {- i}
\absolute {
\frac {\int _U f _ i \rho \,d \nu (u) }{ \int _U \rho \,d \nu (u)} -
\frac {\int _U f _ i \rho \,d \nu' (u) }{ \int_ U \rho \,d \nu' (u)}
}
.\end{equation*}
The space $PM^*_\infty (U)$ depends implicitly on the choice of $\rho$, but we shall keep this dependence implicit in our notation.
We say that a countably generated $\sigma$-algebra $\mathcal{A}$ of subsets of $X$ is \emph{subordinate to $U$ on $Y \subset X$} if for every $x \in Y$ there is some $\delta > 0$ so that
\begin{equation}\label{plaque equation}
B _ \delta ^ U . x \subset [x] _\cA \subset B ^ U _ {\delta ^{-1}} . x
.\end{equation}
If $x$ satisfies~\eqref{plaque equation}, we say that $[x]_\cA$ is a \emph{$U$-plaque for $x$}.

\begin{Proposition} [{\cite[Thm.~6.3 and Thm.~6.29]{Einsiedler-Lindenstrauss-Clay}}]
\label{proposition defining leafwise measures}
Let $X$, $U$ be as above. Then there is a strictly positive function $\rho$ on $U$ so that for every probability measure $\mu$ on $X$ such that the action of $U$ on $(X,\mu)$ is a.e.~free, we have a
Borel measurable map $x \mapsto \mu _ x ^ U$ from $X$ to the space of proportionality classes of Radon measures $P M _ \infty ^{*} (U)$ with the following properties:
\begin{enumerate}
\item there is a co-null set $X '$ so that for every $x \in X '$ and $u \in U$ for which $u . x \in X '$ we have that $\mu _ x ^ {U} = \mu _ {u . x} ^ U u$, with $\mu _ {u . x} ^ U u$ denoting the push forward of $\mu _ {u . x} ^ U$ under right multiplication by $u$.
\item
for a.e.~$x \in X$, we have that $1 \in \supp \mu_x^U$ (i.e.\ the identity is a.s.\ in the support of $\mu_x^U$).
\item 
suppose $\mathcal{A}$ is a countably generated $\sigma$-algebra subordinate to $U$ on $Y \subset X$. For $x \in Y$, let $V_x:=\left\{ u \in U: u . x \in [x] _ \mathcal{A} \right\}$. Then for $\mu$-a.e. $x \in Y$, the conditional measure $(\mu |_ U) _ x ^ {\mathcal{A}}$ on $[x]_{\mathcal{A}}$ is proportional to $(\mu _ x ^ U |_ {V_x}).x$, i.e.~the pushforward of $\mu _ x ^ U |_ {V_x}$ under the map $u \mapsto u.x$.
\item for any $r _ n \uparrow \infty$ and $b _ n > 0$ such that $\sum_ n b _ n ^{-1}< \infty$ we have that
\begin{equation*}
\limsup_ {n \to \infty} \frac {\mu _ x ^ U (B ^ U _{r_n}) }{ b_n \lambda _ U (B ^ U _{r_n+2}) } = 0
.\end{equation*}
 for $\mu$-a.e. $x$.
\end{enumerate}
\end{Proposition}

\noindent
For proof, see~\cite[\S6]{Einsiedler-Lindenstrauss-Clay}. 
Part~(3) is stated in a slightly different but equivalent way in \cite[\S6]{Einsiedler-Lindenstrauss-Clay}, see \cite[\S7.24]{Einsiedler-Lindenstrauss-Clay} for a brief discussion.
Note that by~(4), every choice of $r _ n \uparrow \infty$ and $b _ n > 0$ such that $\sum_ n b _ n ^{-1} < \infty$ it follows that
the function
\begin{equation*}
\rho (u) = \sum_ n \frac {1}{b _ n^2 \lambda _ U (B ^ U _ {r _ n + 2})} 1_{B^U_{r _ n}} (u)
\end{equation*}
is in $L^1(\mu _ x ^ U)$, hence can be used as the $\rho$ defining $PM^*_\infty(U)$. Moreover, the proof actually gives that for any sequence of measurable subsets $B _ n \subset U$
\begin{equation}\label{growth inequality}
\limsup_ {n \to \infty} \frac {\mu _ x ^ U (B_n) }{ b_n \lambda _ U (B ^ U _1 B_n B^U_1) } = 0 \qquad\text{For $\mu$-a.e.~$x$.}
\end{equation}

\medskip

Constructing $\sigma$-algebras which are subordinate to $U$ on a set $Y$ of large measure is not difficult. For instance, it follows from \cite [Cor.~6.15]{Einsiedler-Lindenstrauss-Clay} that for any $\epsilon > 0$, one can find a countably generated $\sigma$-algebra $\mathcal{A}$ which is subordinate to $U$ on $Y$ with $\mu (Y) > 1 - \epsilon$.
The following proposition gives under some extra assumptions  a countably generated $\sigma$-algebra $\mathcal{A}$ subordinate to $U$ on a large subset $Y \subset X$ which plays nicely with that $a$-action:

\begin{Proposition} [{\cite [Prop.~7.36]{Einsiedler-Lindenstrauss-Clay}}]\label{monotone subordinate algebra proposition}
Let $G$, $\Gamma$, be as above $a \in G$ diagonalizable, and $U < G^+ = \{g \in G: a^{-n}ga^n \to 1\}$ closed and normalized by $a$. Let $\mu$ be an $a$-invariant probability measure on $X=G/\Gamma$ (not necessarily ergodic). Then for every $\epsilon > 0$ there is a $a$-invariant subset $Y \subset X$ with $\mu (Y) > 1 - \epsilon$ and a countably generated $\sigma$-algebra $\mathcal{A}$ which is subordinate to $U$ on $Y$ and which is monotonic under $a$ in the sense that $a B \in \cA$ for every $B \in \mathcal{A}$.
\end{Proposition}

Note that $a$-invariance of $Y$ is not explicitly stated in \cite [Prop.~7.36]{Einsiedler-Lindenstrauss-Clay}, but if $\mathcal{A}$ is $a$-monotone and subordinate to $U$ on $Y$ then it is also subordinate to $U$ on \[
Y' = \left (\bigcup_ {k \geq 0} a ^ k Y\right ) \cap \left (\bigcup_ {k \leq 0} a ^ k Y\right ).
\]
Indeed, by monotonicity of $\mathcal{A}$, we have that if $\ell \leq 0 \leq k$,
\begin{equation*}
a ^ {- k} [a^ k x] _ {\mathcal{A}} \subseteq [x] _ {\mathcal{A}} \subseteq a ^ {- \ell} [a^ \ell x] _ {\mathcal{A}}.
\end{equation*}
Thus, if $x \in a ^ \ell Y \cap a ^ k Y$, then as $ [a^ k x] _ {\mathcal{A}}$ contains a small neighborhood of the $U$-orbit around $a^kx$, the atom $[x] _ {\mathcal{A}}$ also contains a small neighborhood of $x$ in its $U$-orbit, and since $[a^ \ell x] _ {\mathcal{A}}$ is a subset of a compact subset of $U$ acting on $a^ \ell x$, the atom $[x] _ {\mathcal{A}}$ is bounded. This shows that indeed $\mathcal{A}$ is subordinate to $U$ on $Y'$. The set~$Y '$ contains~$Y$ and is clearly $a$-invariant up to null sets.

We also note in this context the following:

\begin{Proposition} [{\cite[Lem.~7.16]{Einsiedler-Lindenstrauss-Clay}}] \label{transformation rule for leafwise measures}
Suppose $\langle a \rangle \ltimes U$ acts on $X$ with the action by~$a$ preserving a probability measure $\mu$ on $X$. Then for $\mu$-a.e. $x$
\begin{equation*}
\mu _ {a.x} ^ U = a \mu _ x ^ U a^{-1}
.\end{equation*}
\end{Proposition}

Sometimes it will be convenient to work with countably generated $\sigma$-algebras whose atoms are a bit more general subsets of $U$-orbits. We say that a countably generated $\sigma$-algebra $\mathcal{A}$ of subsets of $X$ is \emph{weakly subordinate to $U$ on $Y \subset X$ relative to $\mu$} if for every $x \in Y$ we have that $[x] _\cA \subset U.x$ and that 
\begin{equation}\label{Vx equation}
V_x = \{u \in U: u.x \in [x]_{\cA}\}
\end{equation}
 is a bounded subset of $U$ with $\mu_x^U(V_x)>0$.

\begin{Lemma} \label{slightly more general sigma-algebra lemma} Let $\mathcal{C}$ be a countably generated $\sigma$-algebra that is weakly subordinate to $U$ on a set $Y \subset X$. Let $V_x$ be as in~\eqref{Vx equation}. Then for $\mu$-a.e.~$x \in Y$,
\begin{equation}\label{conditional measures recipe}
\mu _ x ^ {\mathcal{C}} = \frac 1{\mu _ x ^ {U}(V_x)}\left(\left. \mu _ x ^ {U} \right |_{V_x}.x\right)
.\end{equation}
\end{Lemma}

To prove this, one verifies that  the right-hand side of~\eqref{conditional measures recipe} satsifies the defining properties of the system of conditional measures $\mu _ x ^ {\mathcal{C}}$; we leave the details to the reader.

We will mostly be focusing our attention on the special case where we have a closed subgroup $G < M = \prod_ {i = 1 } ^ m \SL (d, \localfield _ i)$ with $\localfield _ i$ local fields of either zero or positive characteristic, \ $X=G/\Gamma$, \ $J < \Z ^ r$ a finite set, and 
\begin{equation}\label{equation defining U and J}
U = U_J= \bigcap_ {\mathbf n \in J} G ^ + _ \mathbf n.
\end{equation}
 Clearly the coarse Lyapunov groups $U ^ {[\alpha]}$ are of this form. Moreover if $\tilde J$ denotes the set of triplets $(\ell, i, j)$ with $1 \leq \ell \leq m$ and $1 \leq i, j \leq d$ with $i \ne j$ so that the elementary one parameter unipotent subgroups $E_{ij}^ \ell < \SL (d, \localfield _ \ell)$ is contracted by $a ^ {-\mathbf n}$ under conjugation for every $\mathbf n \in J$ then $U_J = G \cap M _ J$ with 
\begin{equation}\label{equation defining M_J}
M _ J = \langle E ^ \ell _ {i, j}: (\ell, i, j) \in \tilde J \rangle
\end{equation}
(note that $M_J$ is not just generated by $E ^ \ell _ {i, j}$ but is in fact equal to the image of the product set $\prod_{(\ell,i,j)\in \tilde J} E ^ \ell _ {i, j}$ under the multiplication map.)

We recall from \cite[\S7]{Einsiedler-Lindenstrauss-Clay} that the measure theoretic entropy $h _ \mu (\mathbf n)$ can be easily obtained from the entropy contributions $D _ \mu (\mathbf n, [\alpha])$ for $[\alpha] \in [\Phi]$ as follows:
\begin{equation}\label{entropy contribution formula}
h _ \mu (\mathbf n) = \sum_ {[\alpha]: \alpha (\mathbf n) > 0} D _ \mu (\mathbf n, [\alpha])
.\end{equation}
By the symmetry of entropy we know that $h _ \mu (- \mathbf n) = h _ \mu (\mathbf n)$ for every $\mathbf n \in \Z ^ r$, hence equation~\eqref{entropy contribution formula} implies that the entropy contributions satisfy the identity
\begin{equation}\label{consequence of entropy symmetry equation}
\sum_ {[\alpha]: \alpha (\mathbf n) > 0} D _ \mu (\mathbf n, [\alpha])=
\sum_ {[\alpha]: \alpha (\mathbf n) > 0} D _ \mu (-\mathbf n, [-\alpha])
.\end{equation}

A relative version of this identity also holds and can be proved along the same lines. Explicitly,
let $\mathcal{A}$ be a countably generated $a ^ {\bullet}$-invariant $\sigma$-algebra (i.e. a $\sigma$-algebra so that for every $B \in \mathcal A$ we have that $a ^ {\mathbf n} B \in \mathcal A$ for every $\mathbf n \in \Z ^ r$). 
We decompose $\mu$ as $\mu = \int \mu _ \xi ^ {\mathcal{A}} \,d \mu (\xi)$ with each $\mu _\xi ^ {\mathcal{A}}$ a probability measure supported on the atom $[\xi]_\mathcal{A}$. Given $[\alpha] \in [\Phi]$ and $\xi \in G / \Gamma$, we can construct a new system of leafwise measures $(\mu _ \xi ^ \mathcal{A}) _ x ^ {[\alpha]}$ (the construction of leafwise measures along $U _ {[\alpha]}$ works for any probability measure on $G / \Gamma$, not necessarily a $a ^ {\bullet}$-invariant one). 
Since the measure $\mu _ \xi ^ {\mathcal{A}}$ is supported on $[\xi] _ \mathcal{A}$, for $\mu _ \xi ^ {\mathcal{A}}$-a.e.\ $x \in G / \Gamma$ we have that $\mu _ \xi ^ {\mathcal{A}}=\mu _ x ^ {\mathcal{A}}$ hence we may define a new system of equivalence classes of up to proportionality of locally finite measures on $U_{[\alpha]}$, to be denoted by $\mu _ x ^ {\mathcal{A},[\alpha]}$, so that for $\mu$-a.e.\ $\xi$, for $\mu_\xi^{\mathcal A}$-a.e. $x$,
\[
\mu _ x ^ {\mathcal{A},[\alpha]} = (\mu _ \xi ^ \mathcal{A}) _ x ^ {[\alpha]}
.\]

For $\mathbf n$ such that $\alpha(\mathbf n)>0$ we  define the \emph{entropy contributions of $[\alpha]$ relative to $\mathcal{A}$} by
\begin{equation}\label{equation defining relative entropy contribution}
\begin{aligned}
D_\mu^{\mathcal{A}}(\mathbf n,[\alpha]) &= \lim_ { \ell \to \infty} \frac {-\log \mu ^ {\mathcal{A},[\alpha]} _ x (a^{-\ell \mathbf n} \Omega _ 0 a^{\ell \mathbf n}) }{ \ell} 
\end{aligned}
\end{equation}
as above we set $D_\mu^{\mathcal{A}}(\mathbf n,[\alpha])=
0$ for 
$\mathbf n$ such that $\alpha(\mathbf n)\leq0$.
These relative entropy contributions are related to the conditional entropy $h _ \mu (\mathbf n| \mathcal{A})$ in the same way that the ordinary entropy contributions
relate to the usual ergodic theoretic entropy:
\begin{equation}
h _ \mu (\mathbf n| \mathcal{A}) = \sum_ {[\alpha]: \alpha (\mathbf n) > 0} D _ \mu ^ {\mathcal{A}} (\mathbf n, [\alpha])
\end{equation}
and satisfy a similar identity to \eqref{consequence of entropy symmetry equation} because of the identity $h _ \mu (\mathbf n| \mathcal{A})=h _ \mu (-\mathbf n| \mathcal{A})$. For properies of conditional entropy, see e.g.\ \cite[\S2]{Einsiedler-Lindenstrauss-Ward-book}.

We note the following observation regarding relative leafwise measures:

\begin{Proposition}\label{relative leafwise measures proposition} Let $[\alpha]$ be a coarse Lyapunov exponent and $\mathcal{A}$ a countably generated $\sigma$-algebra of $U _ {[\alpha]}$-invariant sets. Then $\mu$-a.s.,
\begin{equation}\label{eq:alphaAalphaeq}
\mu _ x ^ {[\alpha]} = \mu _ x ^ {\mathcal{A}, [\alpha]}
\end{equation}
\end{Proposition}

\begin{proof}
Let $Y\subset X$ and $\mathcal C$ be 
 as in Proposition~\ref{monotone subordinate algebra proposition}
 for some $a=a^{\mathbf{n}}$ expanding $U_{[\alpha]}$ and some $\epsilon>0$. 
 By assumption on $\mathcal{A}$, for every $x \in Y$ we have that $[x]_\mathcal{A} \subset [x]_\mathcal{C}$.
 This implies that $(\mu_x^{\mathcal{A}})^{\mathcal{C}}_x=\mu_x^{\mathcal{C}}$
 for a.e.\ $x\in Y$ (see e.g.~\cite[Prop.~5.20]{Einsiedler-Ward-book}). 
 By Proposition~\ref{proposition defining leafwise measures}(3) (applied to
 $\mu$ and $\mu_x^{\mathcal{A}}$) and 
 Proposition~\ref{monotone subordinate algebra proposition}
 it follows from this that for a.e.~$x\in\ Y$ there exists some $r=r(x)>0$
  such that \[\mu_x^{[\alpha]}|_{B_r^{U^{[\alpha]}}}\propto\mu_x^{\cA,[\alpha]}|_{B_r^{U^{[\alpha]}}}.\]
 By Poincar\'e recurrence applied to $a^{-1}$ we have this equation for infinitely many points
 of the form $a^{-n}x\in Y$ and such that the corresponding radii $r(a^{-n} x)$ do not converge to zero. 
 Applying $a^n$ to $a^{-n}x$ and using Proposition \ref{transformation rule for leafwise measures}
 this implies \eqref{eq:alphaAalphaeq} on increasingly larger
 subsets of $U$ for a.e.~$x\in Y$, and the proposition follows.
\end{proof}

\subsection{Product structure of leafwise measures}\label{product structure section}
An important property of how leafwise measures on different course Lyapunov exponents interact is a product structures that is due to A.~Katok and the first named author \cite{Einsiedler-Katok, Einsiedler-Katok-II}.

Consider a group $U_J<G$ constructed from a finite subset $J \subset \Z ^ r$ as in \eqref{equation defining U and J}. Let $[\Phi]_J$ be the collection of course Lyapunov exponents $[\alpha]$ for which $U _ {[\alpha]} \leq U$; clearly 
\[
[\Phi]_J=\{[\alpha]\in[\Phi]: \text{$\alpha(\mathbf n)>0 $ for all $\mathbf n \in J$}\}.
\]
A coarse Lyapunov exponent $[\alpha] \in [\Phi] _ J$ is said to be \emph{exposed} in $U_J$ if there is an element $\mathbf j_{\alpha} \in \Z ^ r$ so that $\alpha (\mathbf j_{\alpha}) \leq 0$ while $\beta (\mathbf j_{\alpha}) > 0$ for all other $[\beta] \in [\Phi] _ J$.
\begin{Lemma}
Set $J'=J\cup \left\{ \mathbf j_{\alpha}\right\}$. Then $U_{[\alpha]} \cap U_{J' }={1}$ and
\begin{equation*}
U_J=U_{[\alpha]}U_{J'} =U_{J '}U_{[\alpha]} 
.\end{equation*}
\end{Lemma}

\begin{proof}
Since $\alpha$ is exposed in $U _ J$ clearly $U_{[\alpha]}\cap U_{J'} =\{1\}$. This also implies that we can find a sequence $\mathbf n _ i \in \Z ^ d$ so that $\alpha (\mathbf n _ i)$ is bounded but $\beta (\mathbf n _ i) \to {-\infty}$ as $i \to \infty$. Recall that $U_J = G \cap M _ J$ with $M _ J < \prod_ {\ell = 1 } ^ m \SL (d, \localfield _ \ell)$ as in \eqref{equation defining M_J}. Let $\tilde J, \tilde J', \tilde J_{[\alpha]}$ denote the set of triplets $(\ell, i, j)$ so that
\begin{align*}
M_J&=\prod_{(\ell,i,j) \in \tilde J} E ^ \ell _ {i j}, \\
M_{J'}&=\prod_{(\ell,i,j) \in \tilde J'} E ^ \ell _ {i j}, \\
M_{[\alpha]}&=\prod_{(\ell,i,j) \in \tilde J_{[\alpha]}} E ^ \ell _ {i j}
.\end{align*}
Then as $[\alpha]$ is exposed, $\tilde J = \tilde J ' \sqcup \tilde J _ {[\alpha]}$ and $M_J=M_{J'}M_{[\alpha]}$. In particular, we may write any $u \in U _ J$ as $u=m'm_{[\alpha]}$ with $m' \in M_{J'}$ and $m_{[\alpha]} \in M_{[\alpha]}$. To prove the lemma one only needs to show $m_{[\alpha]}$ (and hence $m '$) is in $G$. For the sequence $\mathbf n _ i$ described above, $a ^ {\mathbf n _ i} m ' a ^ {- \mathbf n _ i} \to 1$, while $a ^ {\mathbf n _ i} m _ {[\alpha]} a ^ {- \mathbf n _ i}$ is a sequence of elements of $M$ which can be obtained from $m _ {[\alpha]}$ by conjugating with bounded diagonal matrices (in $M$). Without loss of generality, by passing to a subsequence, we can assume
\begin{equation} \label{convergence of subsequence to m sub 1} a ^ {\mathbf n _ i} m _ {[\alpha]} a ^ {- \mathbf n _ i} \to m_1
.\end{equation}
Then $a ^ {\mathbf n _ i} u a ^ {- \mathbf n _ i} \to m_1$, so $m_1 \in G$. However, by equation \eqref{convergence of subsequence to m sub 1} and the fact that while $a ^ {- \mathbf n _ i}$ is an unbounded sequence, on $U_{[\alpha]}$ conjugation by these elements is equivalent to conjugation by bounded elements of the diagonal subgroup of $M$
\begin{equation*}
a ^ {-\mathbf n _ i} m _ 1 a ^ {\mathbf n _ i} \to m_{[\alpha]}
.\end{equation*}
Since $G$ is closed, and both $m _ 1$ and the $a ^ {\mathbf n _ i}$ are in $G$ we conclude that $m_{[\alpha]} \in G$, and the lemma follows.
\end{proof}

The basic phenomena underlying the product structure, which is a slight variation on \cite[Prop.~5.1]{Einsiedler-Katok} and \cite[Prop.~8.5]{Einsiedler-Lindenstrauss-Clay}, is the following:

\begin{Theorem}\label{product structure theorem}
Let $J \subset \Z ^ r $, $ U _ J$, and $[\Phi] _ J$ be as above. Suppose $[\alpha] \in [\Phi] _ J$ is exposed in $U_J$, and set $J'$ as above. Then there is a set of full measure
$X ' \subset X$ so that if $x, u.x \in X '$ for $u=u_{[\alpha]}u' \in U _ J$ (with $u_{[a]} \in U _ {[\alpha]}$ and $u ' \in U _ {J '}$) one has that
\begin{equation*}
\mu _ x ^ {[\alpha]} = \mu _ {u.x} ^ {[\alpha]} u_{[\alpha]}
.\end{equation*}
In particular, for any $[\beta] \in [\Phi] _ J \setminus \{[\alpha]\}$, if $x, u.x \in X '$ for $u \in U_{[\beta]}$ then \[\mu _ x ^ {[\alpha]} = \mu _ {u.x} ^ {[\alpha]}.\]
\end{Theorem}

In order to prove Theorem~\ref{product structure theorem}, we should make use of the system $F_{T,R}$ of subsets of $\Z ^ r$ defined as follows:
\begin{equation}\label{equation defining slices}
F _ {T, R} =
\left\{ \mathbf n \in \Z ^ r: \absolute{\mathbf n} < T,
\absolute {\alpha (\mathbf n)} < R,
\text{and $\beta (\mathbf n) > 0$ for all $\beta \in [\Phi] _ {J '}$}
 \right\}
.\end{equation}

We recall from \cite{Lindenstrauss-pointwise-theorems} that a systems of subsets $\left\{ F_T \right\}_{T \geq T_0}$ of a discrete group $\Lambda$ is said to be \emph{tempered} if there is some $C$ so that for every $T$
\begin{equation}\label{tempered equation}
\absolute {\left (\bigcup_ {T ' < T} F _ {T'} ^{-1} \right) F _ T} < C \absolute {F _ T}.
\end{equation}

We leave the verification of the following easy lemma to the reader:
\begin{Lemma}\label{easy Folner lemma} For any fixed $R > 0$, if $T _ 0$ is large enough:
\begin{enumerate}
\item
The collection $\left\{ F _ {T, R} \right\}_{T \geq T_0}$ defined in \eqref{equation defining slices} is tempered.
\item
There is a $c_R > 0$ so that for $T \geq T _ 0$
\begin{equation*}
c_RT^{r-1} \leq \absolute {F _ {T, R}} \leq 1.01 c_R T ^ {r -1}
.\end{equation*}
\end{enumerate}
\end{Lemma}

We will make use of the following maximal ergodic theorem: 

\begin{Theorem} [{\cite[Thm.~3.2]{Lindenstrauss-pointwise-theorems}}]\label{maximal ergodic theorem}
Let $\Lambda$ be a countable amenable group acting in a measure preserving way on a measure space $(X, \mu)$, and let $\left\{ F _ T \right\} _ {T \geq T _ 0}$ be a tempered sequence of subsets of $\Lambda$. Let $M [f] (x)$ denotes the maximal function
\begin{equation}\label{maximal function definition}
M [f] (x) = \sup_ {T \geq T _ 0} \frac {1 }{ \absolute {F _ T}} \sum_ {h \in F _ T} \absolute {f(hx)}
.\end{equation}
Then there is a constant $C _ 1$ depending only on the constant $C$ in \eqref{tempered equation} so that for any $f \in L ^ 1 _ \mu (X)$,
\begin{equation}\label{maximal inequality}
\mu \left\{ x: M [f] (x) > \lambda \right\} \leq C _ 1 \lambda ^{-1} \norm f _ 1
.\end{equation}
\end{Theorem}

\noindent
Note that $\left\{ F _ T \right\}$ does not need to be a F\o lner sequence for Theorem~\ref{maximal ergodic theorem} to hold.

\begin{proof} [Proof of Theorem~\ref{product structure theorem}]
Let $\epsilon > 0$ be arbitrary. By Lusin's Theorem, there is a compact subset $X _ \epsilon \subset X$ with $\mu (X _ \epsilon) > 1 - \epsilon$ so that the map $x \mapsto \mu _ x ^ {[\alpha]}$ is continuous on $X _ \epsilon$. We may also assume that the subset $X _ \epsilon$ satisfy that if $x, u . x \in X _ \epsilon$ for $u \in U _ {[\alpha]}$ then $\mu _ x ^ {[\alpha]} = \mu _ {u . x} ^ {[\alpha]} u$ and that moreover for every $\mathbf n$
\[
\mu _ {a^{\mathbf n}x} ^ {[\alpha]} = a^{\mathbf n}\mu_x^{[\alpha]}a^{-\mathbf n}.
\]

Fix some $R>0$, and let $M[f]$ denote, for $f \in L ^ 1 (\mu)$, the maximal function for $f$ with respect to averaging on the subsets $\left\{ F _ {T,R} \right\} _ {T \geq T _ 0}$ as in Lemma~\ref{easy Folner lemma}. Let $C _ 1$ be as in \eqref{maximal inequality} for this sequence.
Let
\begin{equation*}
X _ \epsilon ' = X _ \epsilon \cap \left\{ x: M [1 _{X \setminus X _ \epsilon}] (x)\leq 1/8 \right\}
.\end{equation*}
By Theorem~\ref{maximal ergodic theorem}, it follows that
$\mu (X ' _ \epsilon)  \geq 1-(8 C_1 +1) \epsilon$.
Suppose now that $x, u.x \in X ' _ \epsilon$ with $u=u _ {[\alpha]} u ' \in U _ J$.
Then for every $T > 2T _ 0$, we have that the cardinality of $\mathbf n \in F_{T,R}$ for which at least one of $a ^ \mathbf n x, a ^ \mathbf n (ux)$ is not in $X _ \epsilon$ is at most $ \absolute {F_{T,R}}/4$.
Recall that by (2) of Lemma~\ref{easy Folner lemma}, the cardinality of $F_{T,R} \cap \left\{\absolute {\mathbf n} < T/2 \right\} = F_{T/2,R}$ is $\leq 1.01 \cdot 2^{-r+1} \absolute {F_{T,R}}$. Hence for any $T$ large enough there is a $\mathbf n_T \in \Z ^ r$ for which
\begin{enumerate}
\item $\absolute {\alpha (\mathbf n_T)} < R$
\item $a ^ {\mathbf n_T} x, a ^ {\mathbf n_T} (ux) \in X _ \epsilon$
\item for every $\beta \in \Phi$ for which $[\beta] \in [\Phi] _ J \setminus \{[\alpha]\}$ we have that $\beta (\mathbf n_T) < -c T$ for some $c$ independent of $T$
.\end{enumerate}
Set $x_T = a ^ {\mathbf n_{T}}x$, \ $x'_T = a ^ {\mathbf n_{T}}ux$, and suppose $T_j \to \infty$ is such that $(x_{T_j},x'_{T_j})$ converges to say $(x _ \infty, x ' _ \infty)$. Then as $X _ \epsilon$ is compact, $x _ \infty, x ' _ \infty \in X _ \epsilon$.
By assumption, conjugation by both $ a ^ {\pm \mathbf n _ T}$ on $U _ {[\alpha]}$ is an equicontinuous sequence of maps, and on the other hand conjugation by $ a ^ {\mathbf n _ T}$contracts $U _ {J '}$. It follows that w.l.o.g.\ $ a ^ {\mathbf n _ T} u_{[\alpha]} a ^ {- \mathbf n _ T}$ converges along the subsequence $T_j$ to some nontrivial element $\tilde u _ {[\alpha]} \in U _ {[\alpha]}$ while $ a ^ {\mathbf n _ T} u' a ^ {- \mathbf n _ T} \to 1$, hence $x' _ \infty = \tilde u _ {[\alpha]} x _ \infty$. Since $x ' _ \infty, x _ \infty \in X _ \epsilon$ we may conclude that $\mu _ {x _ \infty} ^ {[\alpha]} = \mu _ {x ' _ \infty} ^ {[\alpha]} \tilde u _ {[\alpha]}$.
By continuity of the map $\mu \mapsto \mu _ \mu ^ {[\alpha]}$ on $X _ U$ it follows that as $j \to \infty$ the pairs of proportionality class of measures
\begin{equation*}
\mu _ {a^{\mathbf n _ {T_{\smash j}}}x} ^ {[\alpha]},\quad \mu _ {a^{\mathbf n _ {T_{\smash j}}}x} ^ {[\alpha]} \left (a^{\mathbf n _ {T_{\smash j}}} u_{[\alpha]}a^{-\mathbf n _ {T_{\smash j}}} \right)
\end{equation*}
becomes increasingly similar, hence by equicontinuity of the conjugation by $a^{\pm \mathbf n _ {T}}x$ on $U_{[\alpha]}$
\begin{equation*}
\mu _ x ^ {[\alpha]} = a^{-\mathbf n _ {T_{\smash j}}}\mu _ {a^{\mathbf n _ {T_{\smash j}}}x} ^ {[\alpha]}a^{\mathbf n _ {T_{\smash j}}} \approx \left(a^{-\mathbf n _ {T_{\smash j}}}\mu _ {a^{\mathbf n _ {T_{\smash j}}}ux} ^ {[\alpha]}a^{\mathbf n _ {T_{\smash j}}}\right)u_{[\alpha]} = \mu _ {ux} ^ {[\alpha]}u_{[\alpha]}
\end{equation*}
with the approximation in the middle of the above displayed equation becoming increasingly better as $j \to \infty$. It follows that $\mu _ x ^ {[\alpha]}=\mu _ {ux} ^ {[\alpha]}u_{[\alpha]}$. Taking $X'=\bigcup_{\epsilon} X'_\epsilon$ we obtain the theorem as $\mu(X')=1$.
\end{proof}

\subsection{Behavior of leafwise measures under finite-to-one extensions}

The main result of this subsection is the following:

\begin{Proposition}\label{finite-to-one proposition}
Let $X, X '$ be locally compact, $\pi: X \to X '$ finite-to-one, with a semi-direct product $\langle a \rangle \ltimes U$ acting on both $X$ and $X'$. We assume that $U$ is equipped with a metric $d(\cdot,\cdot)$
inducing the topology on $U$
such that \[cd(u_1,u_2)\leq d(au_1a^{-1},au_2a^{-1})\leq Cd(u_1,u_2)\]
 for some fixed $c,C>1$ and all $u_1,u_2\in U$.
Furthermore we assume that $\pi$ intertwines the action of $\langle a \rangle \ltimes U$ on $X$ and $X '$. 
Let $\mu$ be an $a$-invariant measure on $X$, and let $\mu ' = \pi_*\mu$. Then for $\mu$-a.e. $x \in X$
\begin{equation*}
\mu _ x ^ {U} = (\mu ') ^ {U} _ {\pi (x)}
.\end{equation*}
\end{Proposition}

For simplicity, we assume that the cardinality of the fibers $\pi ^{-1} (x ')$ are the same, say $p$, for all $x ' \in X '$.
This proposition can also be viewed as a special case of the product structure of leafwise measures (cf.~\S\ref{product structure section}): in this case between the conditional measure $\mu$ induces on inverse images $\pi ^{-1} (x ')$ and the leafwise orbits on $U$-orbits.

\begin{proof}
Let $\mathcal B$ denote the Borel $\sigma$-algebra on $X$, and $\mathcal B'$ the Borel $\sigma$-algebra on $X '$ which we identify with the corresponding sub $\sigma$-algebra of $\mathcal B$.
The system of conditional measures $\mu _ x ^ {\mathcal B '}$ can be considered as a measurable map from $X '$ to probability measures on finite subsets of cardinality $p$ of $X$.

Let $\epsilon >0$ be arbitrary, and let $X' _ \epsilon \subset X'$ be a compact subset with $\mu' (X' _ \epsilon) > 1 - \epsilon$ on which the map $x \mapsto \mu _ x ^ {\mathcal B '}$ is continuous.
Let
\begin{equation*}
\bar X' _ \epsilon = \left\{ x' \in X'_\epsilon: \inf_ {N\geq 1} \frac 1N \sum_ {n=0} ^ {N-1} 1_{X _ \epsilon} (a^{-n}.x') \geq 0.9 \right\};
\end{equation*}
by the maximal ergodic theorem, $\mu' (\bar X' _ \epsilon) \geq 1 -10 \epsilon$.

Let $x' \in X '$ and choose $\delta > 0$ and a small open set $B'$ around $x'$ so that:
\begin{enumerate}
\item $\pi ^{-1} (B ') = \bigsqcup_ {i=1} ^ p B_i$, and for each $y' \in B'$ we have that $\absolute {\pi ^{-1} (y ') \cap B _ i} = 1$.
 
\item if $\{x_i\} = B_i \cap \pi^{-1}(x')$,\ $B_\delta^U.x_i  \subset B_i$
\item if $y_i \in B_i$, and $u.y_i \in B_j$ for $u \in U, i\neq j$ then $d(u,1)>100\delta$.
\end{enumerate}
It would be convenient to denote by  $\phi _ i (y ')$ the unique point in $\pi ^{-1} (y ') \cap B _ i$  for $y ' \in B '$.
We claim that if both $y', u.y ' \in B ' \cap \bar X ' _ \epsilon$ for $u \in B^U_\delta$ then
\begin{equation}\label{preimages distribution equation}
\mu _ {y '} ^ {\mathcal B '} (\left\{ \phi _ i (y ') \right\}) = \mu _ {u.y '} ^ {\mathcal B '} (\left\{ \phi _ i (u.y ') \right\}) \qquad i = 1, \dots, p
.\end{equation}
Indeed, since $y', u.y ' \in B ' \cap \bar X ' _ \epsilon$ it follows that there is a subsequence $n _ j \to \infty$ so that both $a ^ {-n_j}.y' \in X'_\epsilon$ and $a ^ {-n_j}.(u.y) \in X'_\epsilon$; moreover since $X'_\epsilon$ is compact we may assume that $y'' = \lim_ {j \to \infty} a ^ {-n_j}.y' = \lim_ {j \to \infty} a ^ {-n_j}.(u.y')$ exists. By continuity of~$\mu _ {\bullet} ^ {\mathcal B '}$ on $X ' _ \epsilon$ it follows that
\begin{equation*}
\mu _ {a ^ {-n_j}.y '} ^ {\mathcal B '}, \ \mu _ {a ^ {-n_j}.(u.y ')} ^ {\mathcal B '} \to \mu _ {y''} ^ {\mathcal B '}
.\end{equation*}
However, by $a$-invariance of $\mu$, the probability vector on the $p$-preimages of a point $z' \in X'$ given by $\mu _ {z '} ^ {\mathcal B '}$ is the same as the probability vector given by $\mu _ {a.y '} ^ {\mathcal B '}$ on the preimages of $a.z'$. This implies that up to permuting the indices on the right-hand side of \eqref{preimages distribution equation}, this equation holds; in view of property~(3) of~$B'$ necessarily this permutation has to be the identity permutation.

Let $\mathcal{A}'_1$ be a $\sigma$-algebra on $X'$ subordinate to $U$ on a set~$Y'$ of measure $\geq 1 - \epsilon_1$.  Refining $\mathcal{A} ' _ 1$ with a finite algebra of sets generated by open balls with $\mu$-null boundaries, we may assume that for every $y' \in B' \cap Y '$ we have that $[y']_{\mathcal{A} ' _ 1} \subset B _ \delta ^ U .y'$.
Now set
\begin{align*}
\mathcal{A} ' & = \left\{ X ' \setminus B '\right \} \cup \left \{ C' \cap B': C' \in \mathcal A_1 ' \right\} \\
\mathcal{A}& = \left\{ X\phantom{'} \setminus B\phantom{'} \right \} \cup \left \{ \pi ^{-1} (C') \cap B _ i: C' \in \mathcal{A} ', 1 \leq i \leq p \right\}
.\end{align*}
For $y' \in B'$, we have that $[y']_{\cA'} = [y']_{\cA'_1} \cap B'$, hence since $B'$ is open, $\mathcal{A} '$ is subordinate to $U$ on $Y' \cap B'$, and by the assumption we made on the atoms $[y']_{\mathcal{A} ' _ 1}$ and property~(3) of~$B'$ we have that if $y' \in B' \cap Y '$ and if we define $C \subset U$ by $[y']_{\cA'} = C.y$ then 
\[
[\phi_i(y')]_{\cA} = C.\phi_i(y') \subset B_i
\]
hence $\cA$ is subordinate to $U$ on $\pi^{-1}(Y' \cap B')$.

Letting $\epsilon \to 0$ we see that \eqref{preimages distribution equation} holds a.e. on $X '$, hence for $y ' \in Y ' \cap B '$ if $\rho$ is the probability measure on $B _ \delta ^ U$ defined by
\begin{equation*}
(\mu') _ {y'} ^ {\cA'} = \rho. y'
\end{equation*}
then
\begin{equation*}
\mu _ {\phi _ i (y')} ^ {\cA}= \rho . \phi _ i (y') \qquad 1 \leq i \leq p
.\end{equation*}
By Proposition~\ref{proposition defining leafwise measures} part~(3), it follows that for a sufficiently small $\delta'$ (possibly depending on $y'$),
\begin{equation*}
\left . (\mu') _ {y'}^U \right| _ {B^U_{\delta'}}= \left . \mu _ {\phi_i(y')}^U \right| _ {B^U_{\delta'}} \qquad 1 \leq i \leq p
\end{equation*}
and hence since $a$ expands $U$ and preserves $\mu$ the proposition follows by Poincar\'e recurrence (cf. Proposition~\ref{transformation rule for leafwise measures}).
\end{proof}

\section{Proof of Theorem~\ref{symmetry theorem}}
Let $[\alpha]$ be a coarse Lyapunov exponent, considered fixed in this section. The key to the proof of Theorem~\ref{symmetry theorem} is a careful comparison between the entropy contributions $D _ \mu (\mathbf n, [\beta])$ and the relative entropy contributions $D _ \mu ^ \mathcal{A} (\mathbf n, [\beta])$ for a specific choice of $\mathcal{A}$, namely the $\sigma$-algebra $\mathcal{A}$ 
corresponding to the Borel map $x \mapsto \mu _ x ^ {[\alpha]}$.
To be precise, we take $\mathcal{A}$ to be the preimage of the Borel $\sigma$-algebra
under the map $x\in X\mapsto \mu _ x ^ {[\alpha]}$, where the leafwise measure
is considered as an element of the space of measures $PM^*_\infty (U_{[\alpha]})$ up to proportionality.
We will fix this $\sigma$-algebra throughout the section.

By Theorem~\ref{product structure theorem}, if $\beta$ is linearly independent from $\alpha$ there is a set $X '\subset G / \Gamma$ of full $\mu$-measure so that for every $x,x' \in X'$ with $x \in U _ {[\beta]}. x '$ we have that $\mu ^ {[\alpha]} _ x = \mu ^ {[\alpha]} _ {x '}$. It follows that there is a countably generated $\sigma$-algebra $\mathcal A '$ equivalent to $\mathcal{A}$ consisting of $U _ {[\beta]}$-invariant sets, and hence by Proposition~\ref{relative leafwise measures proposition}, for every $\beta$ linearly independent from $\alpha$, we have that
\begin{equation*}
\mu _ x ^ {[\beta]} = \mu _ x ^ {\mathcal{A}, [\beta]}\qquad\text{a.e.}
\end{equation*}

It follows from the definition of entropy contribution of coarse Lyapunov exponents \eqref{equation defining entropy contribution} and the definition of relative entropy contribution \eqref{equation defining relative entropy contribution} that for $\beta$ linearly independent from $\alpha$ and $\mathbf n \in \Z ^ r$,
\begin{equation*}
D _ \mu (\mathbf n, [\beta]) = D ^ {\mathcal{A}} _ \mu (\mathbf n, [\beta])
.\end{equation*}
Recall that
\[
 I_x^{[\alpha]}=\{u\in U_{[\alpha]}:\mu _ x ^ {[\alpha]}u = \mu _ x ^ {[\alpha]}\}.
\]

\begin{Lemma}\label{supported on invariance group lemma} With the notations above we have $\supp \mu _ x ^ {\mathcal{A}, [\alpha]}\subset I _ x ^ {[\alpha]}$.
\end{Lemma}

\begin{proof}

By Proposition~\ref{proposition defining leafwise measures}, there is a set $X ' \subset G /\Gamma$ of full $\mu$-measure, so that if $x, u . x \in X '$ for $u \in U _ {[\alpha]}$ then
$\mu _ x ^ {[\alpha]}= \mu _ {u. x} ^ {[\alpha]} u$. Since $\mathcal{A}$ is the $\sigma$-algebra generated by the map $x \mapsto \mu _ x ^ {[\alpha]}$, we may also assume that for every $\xi \in G / \Gamma$, the leafwise measure $\mu ^ {[\alpha]} _ x$ is constant on $[\xi] _ {\mathcal A} \cap X '$.
It follows from $\mu (X ') = 1$ that for $\mu$-a.e. $\xi$ we have that $\mu _ \xi ^ {\mathcal{A} }(X ') = 1$, and also that for $\mu _ \xi ^ {\mathcal{A}}$-a.e. $x$
\begin{equation}\label{equation on leaves}
\mu _ x ^ {\mathcal{A}, [\alpha]} \left ( \left\{ u \in U _ {[\alpha]}: u.x \not\in X ' \cap [x]_{\mathcal A} \right\}\right )=0;
\end{equation}
recall here that by definition
\begin{equation*}
\mu _ x ^ {\mathcal{A}, [\alpha]}  =(\mu _ \xi ^ {\mathcal{A}})^{[\alpha]} _ x \qquad \text{$\mu _ \xi ^ {\mathcal{A}}$-a.s.}
\end{equation*}
hence for $\mu$-a.e.~$\xi$, equation~\eqref{equation on leaves} follows for $\mu_\xi^{\mathcal A}$-a.e.~$x$
from the fact that $X'$ is a conull set with respect to $\mu _ \xi ^ {\mathcal{A}}$.
Fix $x \in X '$ for which \eqref{equation on leaves} holds, and suppose $u \in \supp \mu _ x ^ {\mathcal{A}, [\alpha]}$. Then there exist $u _ i \in U _ {[\alpha]}$ tending to $u$ so that $u _ i . x \in X '$.

By definition of $X '$ it follows that
$\mu _ {u _ i . x} ^ {\mathcal{A}}$ is simultaneously equal to
$\mu _ x ^ {\mathcal{A}}$ and to $\mu _ x ^ {\mathcal{A}}u_i^{-1}$,
hence $u _ i \in I _ x ^ {[\alpha]}$.
Since the latter is a closed subgroup of $U _ {[\alpha]}$, it follows that $u \in I _ x ^ {[\alpha]}$.
\end{proof}

\begin{Lemma}\label{lemma about conditional leafwise measures}
Let $\xi \in G / \Gamma$ and consider the system of leafwise measures along $I_\xi:= I ^ {[\alpha]} _ \xi$
for the probability measure $\mu ^ {\mathcal{A}} _ \xi$, i.e.~$\mu^{\mathcal{A}, I_\xi}_x:= \Bigl(\mu ^ {\mathcal{A}} _ \xi\Bigr)_x^{I_\xi}$.
Then they are a.e. equal to $\mu _ x ^ {\mathcal{A}, [\alpha]}$ in the sense that for every $V _ 1, V _ 2 \subset U _ {[\alpha]}$ and $\mu ^ {\mathcal{A}} _ \xi$-a.e. $x$
\begin{equation*}
\frac
{\mu _ x ^ {\mathcal{A}, [\alpha]} (V _ 1)
}{
\mu _ x ^ {\mathcal{A}, [\alpha]} (V _ 2)}
= \frac
{\mu ^ {\mathcal{A},I_\xi} _ x (V _ 1 \cap I _ \xi)
}{
\mu ^ {\mathcal{A},I_\xi} _ x (V _ 2 \cap I _ \xi)}
.\end{equation*}
\end{Lemma}

\begin{proof}
Let $a_1 = a ^ \mathbf n$ be such that conjugation by $a _ 1$ expands $U _ {[\alpha]}$, i.e.~$\alpha (\mathbf n) > 0$. Let $\epsilon>0$. By Proposition~\ref{monotone subordinate algebra proposition} there is an $a_1$-invariant $Y \subset X$ with $\mu (Y) > 1 - \epsilon$ and an $a_1$-monotone countably generated $\sigma$-algebra $\mathcal{C}$ subordinate to $U _ {[\alpha]}$ on $Y$.

Consider the $\sigma$-algebra $\tilde {\mathcal C}$
\begin{equation*}
\tilde {\mathcal C} = \mathcal C \vee \mathcal A \vee \left\{ X ', X \setminus X ' \right\}
,\end{equation*}
with $X'$ as in the proof of Lemma~\ref{supported on invariance group lemma} and fix $\xi \in Y$ for which $\mu _ \xi ^ {\mathcal{A}} (X ') = 1$.
We claim that $\tilde {\mathcal{C}}$ is weakly subordinate to~$I _ \xi$ on~$Y$ relative to $\mu _ \xi ^ {\mathcal{A}}$. Note that $\mu _ \xi ^ {\mathcal{A}} (X ') = 1$ implies that for $\mu _ \xi ^ {\mathcal{A}}$-a.e. $x$,
\begin{equation}\label{null measure for xi equation}
\mu^{\mathcal{A}, I_\xi}_x \{u \in I_\xi : u.x \in X'\} = 0.
\end{equation}

We first show that for such~$\xi$, for $\mu _ \xi ^ {\mathcal{A}}$-a.e.~$x \in Y$,
\begin{equation}\label{tilde atoms equation}
[x]_{\tilde {\mathcal C}} = [x] _ {\mathcal C} \cap I _ \xi.x \cap X'
.\end{equation}
Indeed, for $\mu _ \xi ^ {\mathcal{A}}$-a.e.~$x$ we have that $x \in X '$ and hence $[x] _ {\tilde {\mathcal C}} \subset X '$.
Moreover, for $x\in Y$ we have that $[x]_{\cC}$ is a $U$-plaque, hence every $z \in [x] _ {\tilde{\mathcal{C}}}$ is of the form $u.x$, and since such $z$ are in particular in $[x]_{\mathcal{A}}$, we know that $\mu _ z ^ {[\alpha]} = \mu _ x ^ {[\alpha]}$.
As in the proof of Lemma~\ref{supported on invariance group lemma}, this implies that $u \in I_x^{[\alpha]}= I_\xi$.
Thus
$[x] _ {\tilde {\mathcal{C}}} \subset [x] _ { {\mathcal{C}}} \cap I _ \xi.x$.
On the other hand, $I _ \xi.x \cap X' \subset [x]_{\mathcal{A}}$, so $I _ \xi.x \cap X' \cap [x]_{\cC} \subset [x]_{\tilde{\cC}}$. This implies~\eqref{tilde atoms equation}.

Let $V_x= \left\{ u \in U: u.x \in [x]_\mathcal{C} \right\}$, 
$\tilde V_x= \left\{ u \in U: u.x \in [x]_{\tilde{\mathcal{C}}} \right\}$. 
If $x \in Y$,
then $V_x$ contains an open neighbourhood $B$ of $1$ in $U_{[\alpha]}$, and by (2) of Proposition~\ref{proposition defining leafwise measures} a.s.~$\mu_x^{I_\xi} (B \cap I_\xi)>0$. Assuming $x$ also satisfies~\eqref{null measure for xi equation} (which again happens a.s.), we have that 
\[
\mu_x^{I_\xi} ((B \cap I_\xi)\setminus \tilde V_x) = 0
\]
hence $\mu_x^{I_\xi}(\tilde V_x)>0$ so $\tilde{\cC}$ is indeed weakly subordinate to $I_\xi$ relative to $\mu _ \xi ^ {\mathcal{A}}$ on a subset of full $\mu_x^{\mathcal{A}}$ of $Y$.
Thus by Lemma~\ref{slightly more general sigma-algebra lemma}, for any bounded $V_1,V_2 \subset U_{[\alpha]}$
\begin{equation*}
\frac
{\Bigl(\mu ^ {\mathcal{A}} _ \xi\Bigr)_x^{I_\xi} (V _ 1 \cap V_x \cap I _ \xi)
}{
\Bigl(\mu ^ {\mathcal{A}} _ \xi\Bigr)_x^{I_\xi} (V _ 2 \cap V_x \cap I _ \xi)}
=
\frac
{\Bigl(\mu ^ {\mathcal{A}} _ \xi\Bigr)_x^ {\tilde {\mathcal{C}}} (V_1.x)}
{\Bigl(\mu ^ {\mathcal{A}} _ \xi\Bigr)_x^ {\tilde {\mathcal{C}}} (V_2.x)}
.\end{equation*}
However, since $\{ X', X \setminus X ' \}$ is a trivial $\sigma$-algebra $\tilde {\mathcal{C}}$ is equivalent to the $\sigma$-algebra $\mathcal{C} \vee \mathcal{A}$ hence
\[
\Bigl(\mu ^ {\mathcal{A}} _ \xi\Bigr)_x^ {\tilde {\mathcal{C}}}=\Bigl(\mu ^ {\mathcal{A}} _ \xi\Bigr)_x^ { {\mathcal{C}}} \qquad\text{a.e.}
\]
By Proposition~\ref{proposition defining leafwise measures} it follows that
\begin{equation*}
\frac
{\Bigl(\mu ^ {\mathcal{A}} _ \xi\Bigr)_x^{[\alpha]} (V _ 1 \cap V_x )
}{
\Bigl(\mu ^ {\mathcal{A}} _ \xi\Bigr)_x^{[\alpha]} (V _ 2 \cap V_x )}
=
\frac
{\Bigl(\mu ^ {\mathcal{A}} _ \xi\Bigr)_x^ { {\mathcal{C}}} (V_1.x)}
{\Bigl(\mu ^ {\mathcal{A}} _ \xi\Bigr)_x^ { {\mathcal{C}}} (V_2.x)}
\end{equation*}
hence
\[
\frac
{\Bigl(\mu ^ {\mathcal{A}} _ \xi\Bigr)_x^{[\alpha]} (V _ 1 \cap V_x )
}{
\Bigl(\mu ^ {\mathcal{A}} _ \xi\Bigr)_x^{[\alpha]} (V _ 2 \cap V_x )}
=
\frac
{\Bigl(\mu ^ {\mathcal{A}} _ \xi\Bigr)_x^{I_\xi} (V _ 1 \cap V_x \cap I _ \xi)
}{
\Bigl(\mu ^ {\mathcal{A}} _ \xi\Bigr)_x^{I_\xi} (V _ 2 \cap V_x \cap I _ \xi)}.
\]
Replacing $\cC$ by $a_1^k \cC$ the sets $V_x$ will contain (for any fixed $x \in Y$) an arbitrary large ball around $1 \in U_{[\alpha]}$, and the lemma follows.
\end{proof}

\begin{Lemma}\label{lemma used to show invariance}
Let $\psi$ be a Borel measurable map from the space of closed subgroups of $U _ {[\alpha]}$ to a symmetric compact subset $\Omega \subset U _ {[\alpha]}$ so that $\psi (I) \in I$ for every closed $I \leq U _ {[\alpha]}$.  Then
\begin{enumerate}
\item for every $f \in L ^\infty (\mu)$ we have that
\begin{equation*}
\int f (x) \,d \mu = \int f (\psi (I _ x ^ {[\alpha]}).x) \,d \mu
.\end{equation*}
\item for $\mu$-a.e.~$x$ it holds that $\psi (I _ x ^ {[\alpha]}).x \in [x] _ {\mathcal{A}}$.
\item for $\mu$-a.e.~$x$ the measure $\mu _ x ^ {\mathcal{A}}$ is $\psi (I _ x ^ {[\alpha]})$-invariant.
\end{enumerate}
\end{Lemma}

\begin{proof}
We start by showing (1). Let $a_1 = a ^ \mathbf n$ be such that conjugation by $a _ 1$ expands $U _ {[\alpha]}$, i.e.~$\alpha (\mathbf n) > 0$, \ $\epsilon > 0$ arbitrary and apply Proposition~\ref{monotone subordinate algebra proposition} to get an $a_1$-invariant $Y \subset X$ with $\mu (Y) > 1 - \epsilon$ and an $a_1$-monotone countably generated $\sigma$-algebra $\mathcal{C}$ subordinate to $U _ {[\alpha]}$ on $Y$.

For any $k$,
\begin{equation} \label{invariance error estimate 1}
\int_ X f (x) \,d \mu - \int_ X f (\psi (I _ x ^ {[\alpha]}) . x) \,d \mu = \int_ X \E_x \left (f(x)- f (\psi (I _ x ^ {[\alpha]}) . x) \middle| a _ 1 ^ k \mathcal C\right ) \!(\xi)\,d \mu(\xi)
.\end{equation}
For $\xi \in Y$ we have that $[\xi] _ {a _ 1 ^ k \mathcal{C}}$ has the form $V(k,\xi) . \xi$ for some $V(k,\xi) \subset U _ {[\alpha]}$, and moreover by (3) of Proposition~\ref{proposition defining leafwise measures}
\begin{equation}\label{invariance error estimate 2}
\begin{aligned}
\E_x \left (f(x)- f (\psi (I _ x ^ {[\alpha]}) . x) \middle| a _ 1 ^ k \mathcal C\right ) \!(\xi)& =
\frac{\displaystyle \int_ {V(k,\xi)} \left (f (u.\xi) - f (\psi (I _ {u.\xi} ^ {[\alpha]}) u. \xi)\right ) \,d \mu _ \xi ^ {[\alpha]} (u)}{\mu _ \xi ^ {[\alpha]}(V(k,\xi))}\\
	&=\frac{\displaystyle \int_ {V(k,\xi)} \left (f (u.\xi) - f (\psi (u I _ {\xi} ^ {[\alpha]} u^{-1}) u. \xi)\right ) \,d \mu _ \xi ^ {[\alpha]} (u)} {\mu _ \xi ^ {[\alpha]}(V(k,\xi))}\\
	&=\frac{\displaystyle \int_ {V(k,\xi)} \left (f (u.\xi) - f (u \psi' (\xi, u). \xi)\right ) \,d \mu _ \xi ^ {[\alpha]} (u)}{\mu _ \xi ^ {[\alpha]}(V(k,\xi))}	
,\end{aligned}
\end{equation}
with $\psi ' (\xi,\bullet )$ some measurable, right $I _ \xi ^ {[\alpha]}$-invariant function $V(k,\xi) \to I _ \xi ^ {[\alpha]} \cap \Omega$.

Define
\begin{equation*}
V ' (k, \xi) = \left\{ u \psi ' (\xi, u): u \in V (k, \xi) \right\}
.\end{equation*}
Then since $I _ \xi ^ {[\alpha]}$ fixes $\mu _ \xi ^ {[\alpha]}$
\begin{equation}\label{using boundary of Greek Omega}
\int_ {V(k,\xi)} f (u \psi' (\xi, u). \xi) \,d \mu _ \xi ^ {[\alpha]} (u) = \int_ {V'(k,\xi)} f (u . \xi) \,d \mu _ \xi ^ {[\alpha]} (u)
.\end{equation}
Set
\begin{equation*}
\partial (\Omega, k) = \left\{ x \in Y: \Omega . x \not\subset [x] _ {a _ 1 ^ k \mathcal{C}} \right\};
\end{equation*}
as $a _ 1$ expands $U _ {[\alpha]}$ and since $\mathcal C$ is subordinate to $U _ {[\alpha]}$ on $Y$, it follows that $\mu (\partial (\Omega, k)) \to 0$ as $k \to \infty$.
On the other hand, by \eqref{using boundary of Greek Omega}, \eqref{invariance error estimate 1} and \eqref{invariance error estimate 2} we obtain
\begin{equation*}
\absolute {\int_ X f (x) \,d \mu - \int_ X f (\psi (I _ x ^ {[\alpha]}) . x) \,d \mu } \leq \norm {f} _ \infty (\mu  (X \setminus Y) + 2 \mu (\partial (\Omega, k)) \leq 2 \epsilon \norm {f} _ \infty
\end{equation*}
for $k$ large enough. Since $\epsilon$ is arbitrary this proves (1).

Let $X '$ be as in (1) of Proposition~\ref{proposition defining leafwise measures} applied to $U _ {[\alpha]}$. Then by (1) we have that
\begin{equation*}
X'' = \left\{ x \in X ': \psi (I _ x ^ {[\alpha]}) . x \in X '\right \}
\end{equation*}
also has full measure. For $x \in X ''$, we have that \begin{equation*}
\mu _ x ^ {[\alpha]} \psi (I _ x ^ {[\alpha]}) ^{-1} = \mu ^ {[\alpha]} _ {\psi (I _ x ^ {[\alpha]}).x}
.\end{equation*}
But $\psi (I _ x ^ {[\alpha]}) \in I _ x ^ {[\alpha]}$, the groups stabilizing from the right $\mu _ x ^ {[\alpha]}$. Hence $\mu _ x ^ {[\alpha]} = \mu ^ {[\alpha]}_ {\psi (I _ x ^ {[\alpha]}).x}$, or equivalently
\[
\psi (I _ x ^ {[\alpha]}).x \in [x] _ {\mathcal{A}}.
\]
This proves (2).

By (1) we have that
\begin{equation} \label{first item equation in long form}
\int f (x) \,d \mu_\xi ^ {[\alpha]} (x) \,d \mu (\xi) = \int f (\psi (I _ \xi ^ {[\alpha]}).x) \, d \mu_\xi ^ {[\alpha]} (x) \,d \mu (\xi)
.\end{equation}
By (2) we see that for a.e.~$\xi$, the measure $\psi (I _ \xi ^ {[\alpha]}).\mu_\xi ^ {[\alpha]}$ is a probability measure on $X$ giving full measure to $[x]_{\mathcal{A}}$. But then \eqref{first item equation in long form} shows that the system of measures $\psi (I _ \xi ^ {[\alpha]}).\mu_\xi ^ {[\alpha]}$ satisfies the defining properties of the conditional measures $\mu_\xi ^ {[\alpha]}$. By uniqueness of conditional measures, it follows that a.s.
\begin{equation*}
\psi (I _ \xi ^ {[\alpha]}).\mu_\xi ^ {[\alpha]} = \mu_\xi ^ {[\alpha]}
,\end{equation*}
establishing (3).
\end{proof}

\begin{Corollary} \label{invariance corollary} The leafwise conditional measures $\mu _ x ^ {\mathcal{A}, [\alpha]}$ are a.s. the proportionality class of the Haar measure on $I ^ {[\alpha]} _ x$.
\end{Corollary}

\begin{proof}
For simplicity, we again denote $I _ x = I _ x ^ {[\alpha]}$. On each item of $\mathcal{A}$ the leafwise measure $\mu _ x ^ {[\alpha]}$ is fixed, hence also $I _ x$. We claim that for $\mu$-almost every~$x$, the conditional measures $\mu _ x ^ {\mathcal{A}}$ is invariant under $I _ x $. Indeed, let $B$ be an arbitrary open subset of $U _ {[\alpha]}$ with compact closure.

By the Borel selector theorem (e.g. \cite [Thm~12.16]{Kechris-book}) there is a Borel 
measurable map $\psi _ B$ from the space of closed subgroups of $U _ {[\alpha]}$ to 
$U _ {[\alpha]}$ so that for subgroups $I$ which are disjoint from $B$ we have that 
$\psi _ B (I) = 1$ whereas for subgroups $I$ intersecting $B$ we have that $\psi _ B (I) \in B$. 
It follows from (3) of Lemma~\ref{lemma used to show invariance} that for $\mu$-almost every 
$x$ for which $B\cap I_x \neq \emptyset$, left multiplication by $\psi _ B (I ^ {[\alpha]} _ {x})$
 preserves $\mu _ x ^ {\cA}$. Thus, by letting $B$ vary under a countable base of the topology 
 of $U_{[\alpha]}$, we get that for a.e.~$x$, the set of $g \in I_x$ preserving 
 $\mu _ x ^ {\cA}$ is dense in $I_x$. Thus, a.s., $\mu _ x ^ {\cA}$ is $I_x$-invariant, 
 hence for $\mu _ x ^ {\cA}$-a.e. $\xi$, the leafwise conditional measure $\mu _ \xi ^ {\cA, I_x}$ 
 is Haar (cf.~\cite[Prob.~6.27]{Einsiedler-Lindenstrauss-Clay}). By Lemma~\ref{lemma about 
 conditional leafwise measures}, it follows that $\mu _ x ^ {\mathcal{A},[\alpha]}$ is a.s. 
 Haar measure on $I_x$.
\end{proof}

\begin{Corollary}\label{main equality corollary}
\[D ^ {\mathcal{A}} _ \mu (\mathbf n, [\alpha]) = \Dinv (\mathbf n, [\alpha])\]
\end{Corollary}

\begin{proof}
This follows immediately from the definitions and Corollary~\ref{invariance corollary}.
\end{proof}

We can now prove Theorem~\ref{symmetry theorem}:
\begin{proof} [Proof of Theorem~\ref{symmetry theorem}]
Let $\mathbf n$ be so that $\alpha(\mathbf n)>0$.
By~\eqref{consequence of entropy symmetry equation},
\begin{equation*}
\sum_ {[\beta]: \beta (\mathbf n) > 0} D _ \mu (\mathbf n, [\beta])=
\sum_ {[\beta]: \beta (\mathbf n) > 0} D _ \mu (-\mathbf n, [-\beta])
.\end{equation*}
Using $h_\mu(\mathbf n |\mathcal A)=h_\mu(-\mathbf n |\mathcal A)$ we get a similar identity
\begin{equation*}
\sum_ {[\beta]: \beta (\mathbf n) > 0} D ^{\mathcal A}_ \mu (\mathbf n, [\beta])=
\sum_ {[\beta]: \beta (\mathbf n) > 0} D ^{\mathcal A} _ \mu (-\mathbf n, [-\beta])
.\end{equation*}
But for $[\beta] \neq [\pm \alpha]$, we have that $D _ \mu (\mathbf n, [\beta]) = D _ \mu ^ {\mathcal{A}} (\mathbf n [\beta])$ hence by combining the above two displayed equations we obtain
\begin{equation*}
D _ \mu (\mathbf n, [\alpha])- D^{\mathcal A} _ \mu (\mathbf n, [\alpha]) = D _ \mu (-\mathbf n, [-\alpha])- D^{\mathcal A} _ \mu (-\mathbf n, [-\alpha])
,\end{equation*}
hence
\begin{align*}
D^{\mathcal A} _ \mu (\mathbf n, [\alpha])&= D _ \mu (\mathbf n, [\alpha]) -D _ \mu (-\mathbf n, [-\alpha])+ D^{\mathcal A} _ \mu (-\mathbf n, [-\alpha]) \\
&\geq D _ \mu (\mathbf n, [\alpha]) -D _ \mu (-\mathbf n, [-\alpha])
\end{align*}
By Corollary~\ref{main equality corollary}, we may conclude that
\begin{equation*}
\Dinv (\mathbf n, [\alpha]) \geq D _ \mu (\mathbf n, [\alpha]) -D _ \mu (-\mathbf n, [-\alpha]).
\end{equation*}
\end{proof}

We concludes this section with a slight extension of Theorem~\ref{symmetry theorem} that will be useful for us in the sequel.

\begin{Theorem} \label{extended symmetry theorem}
Suppose $G$ and $a ^ {\bullet}$ are as in Theorem~\ref{symmetry theorem}. Let $H < F$ be closed subgroups of $G$ with $H \lhd G$ and $F / H$ discrete. Assume furthermore that for any $[\alpha] \in [\Phi]$, the coarse Lyapunov subgroup $U _ {[\alpha]}$ intersects $H$ trivially. Let $\mu$ be an $a^{\bullet}$-invariant and ergodic probability measure on $G/F$. Then for $[\alpha] \in [\Phi]$ and $\mathbf n \in \Z ^ r$ satisfying $\alpha (\mathbf n) > 0$ it holds that

\begin{equation*}
\Dinv (\mathbf n,{[\alpha]}) \geq D _ \mu (\mathbf n,{[\alpha]}) - D _ \mu (-\mathbf n,{[-\alpha]})
.\end{equation*}
\end{Theorem}

Note that under the assumptions of Theorem~\ref{extended symmetry theorem}, for any $[\alpha] \in [\Phi]$ we have that
$U _ {[\alpha]}$ acts locally freely on $G / F$ so that if $\mu$ is an $a ^ {\bullet}$-invariant and ergodic probability measure on $G / F$ we can define $\mu _ x ^ {[\alpha]}$, \ $D_\mu (\mathbf n, [\alpha])$ and $\Dinv (\mathbf n,{[\alpha]})$ just as in the case of $G / \Gamma$ with $\Gamma$-discrete. Essentially this amounts to extending slightly the class of groups we consider: not just closed subgroups of linear groups but quotients of such groups by closed normal subgroups. The proof of the theorem in this case is identical to that of Theorem~\ref{symmetry theorem}. Adapting the results of \cite [\S7]{Einsiedler-Lindenstrauss-Clay} to this setting, in particular the proof of \eqref{entropy contribution formula} also poses no difficulty. We leave the details to the reader.

\section{Proof of Theorem~\ref{refined main theorem}}

A key difference between the positive characteristic and the zero characteristic case is that in zero characteristic unipotent groups have very few subgroups (see \cite{Einsiedler-Lindenstrauss-Mohammadi} for more details). 

Recall that our group $G$ is embedded in $M=\prod_{\ell=1}^m \SL (d, \localfield _ \ell)$, and that $U _ {[\alpha]} = G \cap M _ {[\alpha]}$ with $M _ {[\alpha]}$ a product of unipotent algebraic subgroups of $\localfield _ \ell$.
Enlarging $d$ if necessary, we may (and will) assumes that all the $\localfield _ \ell$ are distinct (recall that we have already assumed that they are all either $\R$ or $\Q _ p$ for appropriate $p$).

The invariance groups $I _ x ^ {[\alpha]}$ are closed subgroup of $M _ {[\alpha]}$. Moreover, since the map $x \mapsto \mu _ x ^ {[\alpha]}$ is measurable (cf.\ Proposition~\ref{proposition defining leafwise measures}) and since the map taking a measure to its right invariance group is also Borel measurable, we have that the map $x \mapsto I _ x ^ {[\alpha]}$ is measurable.
It follows from Proposition~\ref{transformation rule for leafwise measures} that for every $\mathbf n \in \Z^r$,
\begin{equation}\label{invariance transformation equation}
I_{a^{\mathbf n}.x}^{[\alpha]} = a^{\mathbf n} I_x^{[\alpha]} a^{-\mathbf n}
\end{equation}
a.e., hence by ergodicity of $\mu$ these invariance groups are either a.e.\ trivial (i.e.~$= \left\{ 1 \right\}$) or a.e.\ non-trivial.
In \cite{Einsiedler-Katok-II} the following has been shown using Poincare recurrence and the properties of unipotent groups over $\R$ and $\Q_p$:

\begin{Lemma}[{\cite[\S6]{Einsiedler-Katok-II}}]\label{connected algebraic lemma} Under the assumptions of Theorem~\ref{refined main theorem}, for any coarse Lyapunov exponent $[\alpha] \in [\Phi]$, for a.e.~$x$, the group $I _x ^ {[\alpha]}$ can be written as $\prod I_{(x,\ell)}$ with each $I_{(x,\ell)}$ a connected algebraic subgroup of $\SL(d,\localfield_\ell)\cap M_{[\alpha]}$. \end{Lemma}

We exploit this to show that under the assumptions of Theorem~\ref{refined main theorem} the invariance groups of the leafwise measures on the coarse Lyapunov foliations have to be almost everywhere constant:

\begin{Lemma}\label{invariance lemma} Under the assumptions of Theorem~\ref{refined main theorem}, for any coarse Lyapunov exponent $[\alpha] \in [\Phi]$, there is a closed subgroup $I ^ {[\alpha]} \leq U _ {[\alpha]}$ so that $I _ x ^ {[\alpha]} = I ^ {[\alpha]}$ a.e. Moreover, this group is normalized by the group $a^{\bullet}$, and $\mu$ is $I^{[\alpha]}$-invariant.
\end{Lemma}

\begin{proof}
Fix $1 \leq \ell \leq m$. To the unipotent algebraic subgroup $I _ {(x, \ell)}$ of $\SL (d, \localfield _ \ell)$ there corresponds a Lie subalgebra in $\sl (d, \localfield _ \ell)$. By ergodicity and \eqref{invariance transformation equation} it is clear that $\dim (I _ {(x, \ell)})$ is a.e.\ constant, say $k$. Therefore on a set of full measure there is a 1:1 correspondence between $I _ {(x, \ell)}$ and a corresponding homethety class of pure wedge vector $\bar v_{(x, \ell)} \in P(\wedge ^ k \sl (d, \localfield _ \ell))$. The equation \eqref{invariance transformation equation} implies that for every $\mathbf n \in \Z ^ r$
\begin{equation*}
\bar v_{(a ^ {\mathbf n}.x, \ell)} = (\wedge ^ q\! \Ad)(a^{\mathbf n}).\bar v_{(x, \ell)}
.\end{equation*}

By ergodicity of $\mu$, and since $A$ locally compact, to show that $I _ x ^ {[\alpha]}$ is a.e. constant, it is enough to show that for every $a \in A$,
\begin{equation*}
I _ x ^ {[\alpha]} = I _ {a x} ^ {[\alpha]}
.\end{equation*}
Recall that for every $\mathbf n$, the element $a ^ {\mathbf n}$ is of class-$\cA'$.
This implies that $(\wedge ^ q\! \Ad)(a^{\mathbf n})$ is of class- $\mathcal{A} '$. A basic property of elements $g$ of class-$\mathcal{A}'$ is that for any action of the ambient group on a projective space, for any vector $\bar v$ in the project to space, $g ^ k \bar v$ tends to a $g$ invariant point in the projective space (cf.%
\footnote{That proposition deals with a slightly different class of elements that Margulis and Tomanov calls class-$\cA$, but the proof  carries out without any modifications, and indeed whether one uses class-$\cA$ or class-$\cA'$ is purely a matter of taste.}%
~\cite[Prop.~2.2]{Margulis-Tomanov}).
In particular, $(\wedge ^ q\! \Ad)(a^{k\mathbf n}) \bar v _ { (x, \ell)} = \bar v _ {(a^{k\mathbf n}. x, \ell)}$ converges to a $(\wedge ^ q\! \Ad)(a^{k\mathbf n})$-invariant point in the appropriate projective space.
But then Poincare recurrence for the $\Z$-action generated by $a^{\mathbf n}$ implies that
$\bar v _ { (x, \ell)}$ was fixed by $(\wedge ^ q\! \Ad)(a^{\mathbf n})$ to begin with.
Since $\bar v _ { (x, \ell)}$ uniquely determines the algebraic group $I _ {(x, \ell)}$ it follows that
\begin{equation*}
I _ {(x, \ell)} = I _ {(a^{\mathbf n}.x, \ell)}
\end{equation*}
and since this is true for every $\ell$ it follows that $I_x ^ {[\alpha]} = I _ {a ^ {\mathbf n} . x} ^ {[\alpha]}$ and the first claim of the lemma follows.

To show invariance of $\mu$ under $I ^ {[\alpha]}$ we apply Corollary~\ref{invariance corollary}. Let $\mathcal{A}$ be a $\sigma$-algebra corresponding to the Borel map $x \mapsto \mu _ x ^ {[\alpha]}$.
According to the lemma, for $\mu$-almost every $x$, the conditional measure $\mu ^ {\mathcal{A}} _ x$ is a probability measure with the properties that for $\mu ^ {\mathcal{A}} _ x$-a.e. $\xi$ the leafwise measure
$\mu ^ {\mathcal{A}, [\alpha]} _ \xi = \left (\mu ^ {\mathcal{A}} _ x \right) ^ {[\alpha]} _ \xi$
is the Haar measure on $I ^ {[\alpha]}$.
By Lemma~\ref{lemma about conditional leafwise measures} and \cite[Prop.~4.3]{Lindenstrauss-quantum} (or \cite[Prob.~6.27]{Einsiedler-Lindenstrauss-Clay}) it follows that $\mu ^ {\mathcal{A}} _ x$ is $I ^ {[\alpha]}$-invariant, hence so is $\mu = \int \mu ^ {\mathcal{A}} _ x \,d \mu (x)$.
\end{proof}

\Head{Remark}
We note that the same argument will work also in the slightly more general setup of $\mu$ an $a^{\bullet}$-invariant and ergodic probability measure on $G/F$ where $F$ is as in Theorem~\ref{extended symmetry theorem}, as long as $\localfield_\ell$ are all $\R$ or $\Q_p$ (for possibly more than one choice of $p$) and $a^\bullet$ satisfies the class-$\cA'$ assumption.

\medskip

Let $J_u$ be the closed group generated by all one-parameter unipotent groups preserving $\mu$.  Let $J$ be the group generated by $J_u$ and $a^{\bullet}$. Note that by Lemma~\ref{connected algebraic lemma} for all $[\alpha] \in [\Phi]$ the group $I^{[\alpha]}$ is generated by unipotent one parameter groups, hence by Lemma~\ref{invariance lemma} we have that $ I ^ {[\alpha]} \leq J_u$.

\begin{Lemma}\label{Margulis-Tomanov application lemma} Under the assumptions of Theorem~\ref{refined main theorem}, there are closed subgroups $H,L$ of $G$ with $H \lhd L$, \ $J_u \leq H$ and $J \leq L$ so that
\begin{enumerate}
\item $\mu$ is $H$-invariant
\item every $x \in \supp \mu$ has a periodic $H$-orbit
\item $\mu$ is supported on a single $L$-orbit.
\end{enumerate}
\end{Lemma}

\begin{proof}
By definition, the measure $\mu$ is $J$-invariant. Since $\mu$ is $a^ {\bullet}$-ergodic and $a ^ {\bullet} \leq J$, it follows that $\mu$ is $J$-ergodic. The statement now follows from the main result of Margulis and Tomanov's paper \cite{Margulis-Tomanov-2}.\end{proof}

\noindent
 We note that the main ingredient used by Margulis and Tomanov in \cite{Margulis-Tomanov-2} is a measure classification result \cite{Margulis-Tomanov, Ratner-padic} extending Ratner's Measure Classification Theorem \cite{Ratner-Annals, Ratner-Acta} to the $S$-arithmetic setting.

 \medskip

\begin{proof} [Proof of Theorem~\ref{refined main theorem}]
Without loss of generality we may (and will for the remainder of this section) assume $[1]_\Gamma \in \supp \mu$, and since $\mu$ is supported on a single $L$-orbit we may as well assume $L=G$ and $L>\Gamma$. Also, by Theorem~\ref{symmetry theorem} we have that if $J_u$ is trivial, equation \eqref{quotient symmetry equation} holds for $H=\{1\}$, hence we may assume that $J_u$ is non-trivial.
By Lemma~\ref{Margulis-Tomanov application lemma} we know that $H.[1]_\Gamma$ is periodic, i.e. that $H \cap \Gamma$ is a lattice in $H$. By \cite[Thm.~1.13]{Raghunathan-book} we have that $H\Gamma$ is closed, hence if $\pi:L\to L/H$ is the natural projection, $\Lambda = \pi(H\Gamma)=\pi(\Gamma)$ is a discrete subgroup of $L/H$.

Let $H _ u$ denote the subgroup of $H$ generated by one parameter unipotent groups. It is clear from the definition that since $H \lhd L$ conjugation by elements of $L$ preserves the class of unipotent one parameter subgroups of $H$, hence $H_u \lhd L$.
Recall that $J _ u \leq H _ u$, in particular by assumption $H _ u$ is nontrivial. From the definition it is clear $H _ u$ has the form of a product of (possibly trivial) subgroups $H_{u,\ell} < \SL (d, k_\ell)$ for $1 \leq \ell \leq m$, with each $H_{u,\ell}$ generated by one parameter unipotent groups. These $H_{u,\ell}$ are essentially algebraic: if $ M _ \ell$ is the Zariski closure of $H _ {u,\ell}$, then the radical of $M _ \ell$ equals its unipotent radical (otherwise $M_\ell$ would have contained an algebraic proper subgroup that contain all the unipotent elements of $M_\ell$ --- in contradiction to the definition of $M_\ell$). Moreover this also implies that the Lie algebras of $H _ {u,\ell}$ and $M _ \ell$ coincide, hence we have that $H _ {u,\ell} = M _ \ell ^ +$ --- the closed subgroup of $M _ \ell$ generated by one parameter unipotent subgroups\footnote{We use the notation $M^+$ only for algebraic groups $M$.}. It follows from \cite[Thm.~I.2.3.1]{Margulis-book} that $H _ {u,\ell}$ has finite index in~$M _ \ell$.%

Let $M=\prod_\ell  M_\ell$. Since $L$ normalizes $H _ u$ it also normalizes its Zariski closure $M$, hence $L$ is a subgroup of the normalizer $N = \prod_ \ell N _ \ell $ of $M$ in $\prod _{\ell = 1 } ^ m \SL (d, \localfield _ \ell)$. As quotients of algebraic groups we can embed for every $\ell$ the group $N _ \ell/M _ \ell$ into some $\SL (d _ \ell, \localfield _ \ell)$, and taking $d'=\max(d_\ell)$ we can therefor view $L/M$ as a closed subgroup of $\prod_\ell \SL(d',\localfield_\ell)$.
Let $\pi _ u$ denote the natural projection map $L \to L/H_u$ and $\pi _ M$ the natural projection $N \to N / M$. Then the induced $ a ^ {\bullet}$-action on $(L/H)/\Lambda$ is isomorphic to the induced $a ^ {\bullet}$-action on $(L/H_u)/\Lambda '$ with $\Lambda '$ the closed subgroup of $L / H _ u$ given by $\pi _ u (\Gamma H)$. The group $\Lambda '$ is not necessarily discrete, but is discrete modulo the normal subgroup $H / H _ u$ of $L / H _ u$.

Since $H_u$ has finite index in $M$, the space $(L/H_u)/\Lambda '$ is in turn a finite extension of $(L/M)/\Lambda ''$ with $\Lambda '' = \pi _ M (\Gamma H)$. The induced action of $a ^ {\bullet}$ on $(L/M)/\Lambda ''$, equipped with the measure $\mu '' = (\pi_M) _* \mu$ satisfies the condition of Theorem~\ref{extended symmetry theorem}, and the coarse Lyapunov subgroups of $L / M$ coincide with the images of the coarse Lyapunov subgroups $U _ {[\alpha]}$ of $L$ under $\pi _ M$.

The measure $\mu ''$ cannot be invariant under any one parameter unipotent subgroup of $L/M$, for if $u'_t = \exp (\mathbf u' t) $ were such a subgroup with $\mathbf u' \in \Lie (L / M)$ nilpotent then in view of the fact that $M$ is algebraic we can find a nilpotent $\mathbf u$ so that $d \pi _ M (\mathbf u) = \mathbf u '$. The invariance of $\mu''$ under $u'_t$ implies the same for $\mu ' = (\pi _ u) _ {*} \mu $ on $(L / H _ u) / \Lambda '$. Since $\mu$ is $H _ u$-invariant, it follows that $\mu$ will be invariant under the one parameter unipotent subgroup $u _ t$ of $L$. But if this were so, $u _ {\bullet}$ would have been contained in $J _ u$, hence in $H _ u$, hence $\mathbf u '$ would be zero --- in contradiction.

Therefore applying Theorem~\ref{extended symmetry theorem}, we conclude that for every coarse Lyapunov exponent $[\alpha]$
\begin{equation*}
D _ { \mu ''} (\mathbf n, [\alpha]) = D _ { \mu ''} (-\mathbf n, [-\alpha])
.\end{equation*}
By Proposition~\ref{finite-to-one proposition} this property is preserved by finite-to-one extensions, hence
\begin{equation*}
D _ {\mu '} (\mathbf n, [\alpha]) = D _ {\mu '} (-\mathbf n, [-\alpha]).
\end{equation*}
This identity implies the theorem in view of the isomorphism between the $\pi \circ a ^ {\bullet}$-action on $(L/H)/\Lambda$ and the $\pi _ u \circ a ^ {\bullet}$-action on $(L/H_u)/\Lambda '$.
\end{proof}

\section{Proof of Theorem~\ref{affine group theorem}}
Recall the notations in Theorem~\ref{affine group theorem}; in particular, $G = \SL (n, \R) \ltimes \R ^ n$, \ $\Gamma = \SL (n, \Z) \ltimes \prod \Z ^ n$, \ $A$ is the maximal diagonalizable subgroup of $\SL (n, \R) < G$, and $\mu$ is an $A$-invariant and ergodic measure on $G / \Gamma$.
Let $\pi$ denote the natural map $G \to \SL(n,\R)$ as well as the induced map $G / \Gamma \to \SL (n, \R) /\SL(n,\Z)$.
Let $\bar \mu = \pi_*\mu$. There are two cases:
\begin{itemize}
\item \textbf{Positive base entropy}: there is some $a \in A$ for which $h _ {\bar \mu} (a) > 0$
\item \textbf{Zero base entropy}: $h _ {\bar \mu} (a) = 0$ for every $a \in A$.
\end{itemize}

The first case has essentially been already been taken care of in \cite{Einsiedler-Katok-Lindenstrauss} and \cite{Einsiedler-Lindenstrauss-joinings-2}. Indeed, by \cite{Einsiedler-Katok-Lindenstrauss} in this case the measure $\bar \mu$ is homogeneous. Moreover one can explicitly list the possible $A$-invariant and ergodic homogeneous measures (\cite[\S6]{Lindenstrauss-Barak}; cf.~also \cite{Einsiedler-Lindenstrauss-Michel-Venkatesh} for a related discussion): when $n$ is a prime, the only possibility is that $\bar \mu$ is Haar measure on $\SL (n, \R) / \SL (n, \Z)$. When $n$ is not a prime, there can be additional intermediate cases, corresponding to degree $d$ totally real extensions $K$ of $\Q$ for $d \mid n$. More explicitly, it follows from the results of \cite[\S6]{Lindenstrauss-Barak}, that such measures are supported on an orbit $L.[g_1]$ of the reductive group
\begin{equation*} L = \left (\prod_ {i=1}^{n} \GL (n/d, \R) \right) \cap \SL (n, \R)
\end{equation*}
with $g_1 \Z ^ d$ homothetic to a finite index sublattice of the lattice $\mathcal{O}_K \otimes \Z^{n/d}$,  where $\mathcal{O}_K$ is the ring of integers of the totally real field $K$, and we view $\mathcal{O}_K$ as a lattice in $\R ^ d$ in the usual way --- i.e. if $\tau _ 1, \dots, \tau _ d$ are the $d$ distinct field embeddings of $K$ in $\R$, we identify $\mathcal{O}_K$ with the lattice
\begin{equation*} \left\{ (\tau _ 1 (\mathbf n), \dots, \tau _ d (\mathbf n)): \mathbf n \in \mathcal{O} _ K \right\} .
\end{equation*}
Note that the case $d=n$ is also meaningful, but corresponds to the case of $\bar \mu$ being the natural measure on an $A$-periodic orbit which is excluded according to our entropy assumption.

Let $L _ 1 = [L, L] = \prod_ {i = 1 } ^ {n / d} \SL (n/d, \R)$, and let $A_1 = A \cap U _ 1$ and $A _ 2 = A \cap C_G(L)$ (with $C_G(L)$ the centralizer of $L$ in $G$). Then for every $x \in L.[g_1] = \supp \bar \mu$, we have that $A_2 .x $ is periodic. The stabilizer $\Delta$ of $x$ in $A_2$ does not depend on $x$ --- indeed, it is given by $g_1 ^{-1} A_2 g_1 \cap \SL (n, \Z)$ and is commensurable to the image of $\mathcal{O} _ K ^ {\times}$ under the map $\mathbf n \mapsto \operatorname{diag} (\tau _ 1 (\mathbf n), \dots, \tau _ d (\mathbf n))$. Moreover $\bar \mu$ has a very simple ergodic decomposition with respect to $A _ 1$: if $\bar \mu _ 1$ is the uniform measure on the periodic orbit $L_1[g_1]$ then $\bar \mu = \int_ {A _ 2 / \Delta} h. \bar \mu _ 1 \,d h$; by the Howe-Moore ergodicity theorem $\bar \mu _ 1$ is $A _ 1$-ergodic.

Since the case of $d=n$ has been excluded, the group $ \tilde L _ 1 = g _ 1 ^{-1} L _ 1 g _ 1\ltimes  \R ^ n$ is a perfect algebraic group. Since $L_1[g_1]$ was periodic in $\SL(n,\R)/\SL(n,\Z)$, the group $\tilde L_1$ is also defined over $\Q$, and $\Gamma \cap L_1$ is an arithmetic lattice in it. Letting $\mu = \int \mu _ \xi \,d \xi$ denote the ergodic decomposition of $\mu$ with respect to $A _ 1$, it is clear that for a.e.~$\xi$ there would be some $h \in A _ 2$ (unique up to $\Delta$) so that $\pi _ {*} (\mu _ \xi) = h . \bar \mu _ 1$. We can now apply \cite[Thm.~1.6]{Einsiedler-Lindenstrauss-joinings-2} to conclude that the measures $\mu_\xi$ are all homogeneous, i.e. supported on a single orbit of a group $M \leq \SL (n, \R) \ltimes \R ^ n$ with $\pi (M) = L_1$.
Moreover e.g.~by Poincar\'e recurrence (or by noting that $ g _ 1 ^{-1} L _ 1 g _ 1$ acts irreducibly on $\R^n/\Z^n$, hence either $\pi$ is injective on $M$ or $M=\SL (n, \R) \ltimes  \R ^ n$) the group $M$ is normalized by $A_2$, so $\bar \mu$ is a homogeneous measure supported on a single $A_2 M$-orbit. Note that we have essentially classified which homogeneous measures may occur in the positive base entropy case.

There remains the zero base entropy case. We will parameterize the acting group by vectors $\mathbf m \in \Z ^ n$ with $\sum _ i m _ i = 0$, and set
\begin{equation*}
a ^ {\mathbf m} = \begin{pmatrix} e^{m_1}& \\ & e^{m_2} \\&&\ddots \\&&&e^{m_n} \end{pmatrix}
.\end{equation*}
There are exactly $n^2$ coarse Lyapunov exponents $[\alpha]$ that play a role for the action of $a ^ {\bullet}$ on $G / \Gamma$ (i.e.~for which $U_{[\alpha]}$ is nontrivial), for each of which $\dim (U_{[\alpha]})=1$:
$n(n-1)$ coarse Lyapunov exponents for which
$U_{[\alpha]} < \SL (n, \R) $, corresponding to the functionals $\phi_{i,j}:\mathbf m \mapsto m_i - m_j$ for $i \neq j$, and $n$ coarse Lyapunov exponents for which $U_{[\alpha]}$ is a subgroup of the unipotent radical of $G$ corresponding to the functionals $\phi_i:\mathbf m \mapsto m_i$.

By \cite [Prop.~6.4]{Einsiedler-Lindenstrauss-joinings-2}, if $D _ \mu (\mathbf m, [\phi_{i,j}]) \neq 0$ for some $i \neq j$ then $D _ {\bar \mu} (\mathbf m, [\phi_{i,j}]) \neq 0$ and hence we have positive base entropy. Therefore we may assume $D _ \mu (\mathbf m, [\phi_{i,j}])=0$ for all $i$, $j$.

It follows that only $D _ \mu (\mathbf m, [\phi_{i}])$ for $i=1,\dots,n$ can be nonzero, and since for some $\mathbf m$ we have that $h _ \mu (\mathbf m)>0$, by \eqref{entropy contribution formula} for at least one $i$ we have that $D _ \mu (\mathbf m, [\phi_{i}])>0$. However, the only way to satisfy \eqref{consequence of entropy symmetry equation} for all relevant\footnote {Recall that for convenience we take $\mathbf m \in \Z ^ n$ with $\sum m _ i = 0$.} parameters $\mathbf m$ is if for \emph{all} $i$ we have $D _ \mu (\mathbf m, [\phi_{i}])>0$ (indeed these contributions must equal to one another, though we will not need to make use of that). 
The key point is that unlike the $[\phi_{i,j}]$, for $[\phi_i]$ the opposite coarse Lyapunov exponent $[-\phi_i]$ does not appear in $G$ (in other words, $U_{[-\phi_i]}$ is trivial), hence by Theorem~\ref{symmetry theorem} the invariance group $I _ \mu ^ {[\phi_i]}$ is nontrivial\footnote{In Theorem~\ref{symmetry theorem} the measure $\mu$ is assumed to be ergodic under $a^{\bullet}$, whereas we only know $\mu$ is ergodic under $A$. Since $a^{\bullet}$ is cocompact in $A$ this makes very little difference; in particular, a reader worried about this slight discrepancy may take the ergodic decomposition of $\mu$ with respect to the action of $a^{\bullet}$, and prove invariance of $\mu$ under the unipotent radical of $G$ by showing it on each ergodic component seperately.}. In this case the only nontrivial closed subgroup of $U_{[\phi_i]}$ with arbitrarily small and arbitrarily large elements is $U_{[\phi_i]}$ itself, hence as in the proof of Theorem~\ref{refined main theorem} the measure $\mu$ is invariant under the group generated by all the $U_{[\phi_i]}$, i.e.~by the unipotent radical of $G$, which implies the second case of Theorem~\ref{affine group theorem}.


\def\cprime{$'$}

\end{document}